\documentclass{amsart}

\usepackage{amsmath}
\usepackage{amsthm}
\usepackage{amsfonts}
\usepackage{amssymb}

\newcommand\R{{\mathbf{R}}}
\newcommand\C{{\mathbf{C}}}

\newcommand\E{{\mathbf{E}}}
\newcommand\Z{{\mathbf{Z}}}
\newcommand\F{{\mathbf{F}}}

\newcommand\eps{\varepsilon}

\newcommand\dist{\operatorname{dist}}

\theoremstyle{plain}
  \newtheorem{theorem}[subsection]{Theorem}

  \newtheorem{proposition}[subsection]{Proposition}
  
  \newtheorem{lemma}[subsection]{Lemma}
  \newtheorem{corollary}[subsection]{Corollary}

\theoremstyle{remark}
  \newtheorem{remark}[subsection]{Remark}
  
  \newtheorem{example}[subsection]{Example}
  \newtheorem{examples}[subsection]{Examples}

\theoremstyle{definition}
  \newtheorem{definition}[subsection]{Definition}

\include{psfig}

\begin{document}

\title{Freiman's theorem for solvable groups}

\author{Terence Tao}
\address{Department of Mathematics, UCLA, Los Angeles CA 90095-1555}
\email{tao@math.ucla.edu}

\begin{abstract}  Freiman's theorem asserts, roughly speaking, if that a finite set in a torsion-free abelian group has small doubling, then it can be efficiently contained in (or controlled by) a generalised arithmetic progression.  This was generalised by Green and Ruzsa to arbitrary abelian groups, where the controlling object is now a coset progression.  We extend these results further to solvable groups of bounded derived length, in which the coset progressions are replaced by the more complicated notion of a ``coset nilprogression''.  As one consequence of this result, any subset of such a solvable group of small doubling is is controlled by a set whose iterated products grow polynomially, and which are contained inside a virtually nilpotent group.  As another application we establish a strengthening of the Milnor-Wolf theorem that all solvable groups of polynomial growth are virtually nilpotent, in which only one large ball needs to be of polynomial size.  This result complements recent work of Breulliard-Green, Fisher-Katz-Peng, and Sanders. 
\end{abstract}

\maketitle

\section{Introduction}

Define an \emph{additive set} to be a pair $(A,G)$, where $G = (G,+)$ is an abelian group and $A$ is a finite non-empty subset of $G$; we shall often abuse notation and refer to $(A,G)$ just as $A$.  Given two additive sets $(A,G), (B,G)$, we define the sumset $A+B := \{a+b: a \in A, b \in B\}$ and difference set $A-B := \{a-b: a \in A, b \in B\}$, as well as the reflection $-A := \{-a: a \in A \}$.  We also define the iterated sumset $kA := A + \ldots + A$ for $k \geq 1$ (where $k$ summands appear on the right-hand side), with the convention $0A = \{0\}$, and the iterated sum-and-difference-set 
$$\pm k A := k(-A \cup \{0\} \cup A) = \bigcup_{j,j' \geq 0: j+j' \leq k} jA - j'A.$$

Define the \emph{doubling constant} of an additive set $(A,G)$ to be the ratio $|2A|/|A|$, where $|A|$ denotes the cardinality of $A$.  We have the well-known inverse theorem of Freiman \cite{frei} that classifies (up to constants) the sets of small doubling in the integers:

\begin{theorem}[Freiman's theorem in $\Z$]\cite{frei} Let $(A,\Z)$ be an additive set of integers of doubling constant at most $K$ for some $K \geq 1$.  Then there exists a generalised arithmetic progression $P$ of rank $O_K(1)$ and cardinality $|P| \ll_K |A|$ which contains $A$.
\end{theorem}

See Section \ref{notation-sec} below for the asymptotic notation used in this paper, and Appendix \ref{pse} for the definition of a generalised arithmetic progression.

In \cite{gr-4}, Green and Ruzsa generalised Freiman's theorem to arbitrary abelian groups:

\begin{theorem}[Freiman's theorem in abelian groups]\label{gr}\cite{gr-4} Let $(A,G)$ be an additive set of doubling constant at most $K$ for some $K \geq 1$.  Then there exists a coset progression $H+P$ of rank $O_K(1)$ and cardinality $|H+P| \ll_K |A|$ which contains $A$.  (See Appendix \ref{pse} for the definition of a coset progression.)  Furthermore we have $H+P \subset \pm O_K(1) A$.
\end{theorem}

This result was deduced from the following closely related fact (see \cite[Section 5]{gr-4}; the case $G=\Z$ was treated earlier in \cite{ruzsa}, \cite{chang}):

\begin{theorem}[Ruzsa-Chang's theorem in abelian groups]\label{gr2}\cite{gr-4} Let $(A,G)$ be an additive set of doubling constant at most $K$ for some $K \geq 1$.  Then $\pm 4A$ contains a coset progression $H+P$ of rank $O_K(1)$ and cardinality $|H+P| \gg_K |A|$.
\end{theorem}

Indeed, one easily verifies (using the Ruzsa-Pl\"unnecke sum set estimates, see e.g. \cite[Chapter 2]{tao-vu}) that an additive set $(A,G)$ of doubling constant at most $K$ can be covered by $O_K(1)$ translates (by elements in $\pm 5A$) of the coset progression $H+P \subset \pm 4A$ identified in Theorem \ref{gr2}, and the union of these translates can be contained in a slightly larger coset progression in $\pm O_K(1) A$.  (See also \cite[Section 6]{bg} for further discussion.)

Now we consider the analogous situation for non-abelian groups.  Define a \emph{multiplicative set} to be a pair $(A,G)$ where $G = (G,\cdot)$ is a multiplicative group and $A$ is a finite non-empty subset of $G$.  We can define the product set $A \cdot B := \{ab: a \in A, b \in B \}$, reflection $A^{-1} := \{a^{-1}: a \in A \}$, the iterated product set $A^k := A \cdot \ldots \cdot A$, iterated product-quotient-set 
$$A^{\pm k} := (A \cup \{1\} \cup A^{-1})^k,$$
doubling constant $|A^2|/|A|$, and tripling constant $|A^3|/|A|$, in exact analogy with the additive case.  
We adopt the notation $\prod_{i=1}^n A_i := A_1 \cdot \ldots \cdot A_n$ and $\prod_{i=n}^1 A_i := A_n \cdot \ldots \cdot A_1$ for multiplicative sets $(A_1,G),\ldots,(A_n,G)$; note the order of multiplication is important when $G$ is not abelian.  

It is also convenient to introduce the following definition.

\begin{definition}[$K$-control]  Let $K \geq 1$.  A multiplicative set $(A,G)$ is said to be \emph{$K$-controlled} by another multiplicative set $(B,G')$ if $G'$ is a subgroup of $G$, if $|B| \leq K |A|$, and there exists another finite non-empty subset $X$ of $G$ with $|X| \leq K$ such that $A \subset (X \cdot B) \cap (B \cdot X)$.
\end{definition}

\begin{remark}\label{control}  Observe that if $(A,G)$ is $K$-controlled by $(B,G')$ and $B$ has doubling constant at most $K$, then
$$|A \cdot A| \leq |X \cdot B \cdot B \cdot X| \leq K^2 |B \cdot B| \leq K^3 |B| \leq K^4 |A|.$$
Thus small doubling of $A$ is implied by small doubling of $B$ if $A$ is controlled by $B$.  We also observe a transitivity property: if $(A,G)$ is $K$-controlled by $(B,G')$, and $(B,G')$ is $K'$-controlled by $(C,G'')$, then $(A,G)$ is $KK'$-controlled by $(C,G'')$.
\end{remark}

The coset progression $H+P$ obtained in Theorem \ref{gr2} can be easily shown to $O_K(1)$-control $A$ (viewing the ambient group $G$ multiplicatively).
Thus we see that in the abelian case, sets of small doubling are controlled by coset progressions.  It is of interest to understand what the analogous controlling sets are in non-abelian groups.

Our results in this paper will mostly be restricted to solvable groups.  A solvable group which is particularly easy to analyse (and which serves as a simple model case for our more general results) is the \emph{lamplighter group}:

\begin{definition}[Lamplighter group]\label{lamp-def}  Let $\F_2^\Z$ be the additive group $\F_2^\Z := \{ (x_i)_{i \in \Z}: x_i \in \F_2 \hbox{ for all } i \in \Z \}$, where $\F_2$ is the field of two elements, and let $(\F_2^\Z)_0$ be the finitely generated subgroup of $\F_2^\Z$ consisting of those sequences $(x_i)_{i \in \Z}$ with at most finitely many $x_i$ non-zero.  Let $T: \F_2^\Z \to \F_2^\Z$ be the shift automorphism $T(x_i)_{i \in \Z} := (x_{i-1})_{i \in \Z}$; this preserves $(\F_2^\Z)_0$. The \emph{extended lamplighter group} $\Z \ltimes \F_2^\Z$ is the set $\{ (n, x): n \in \Z, x \in \F_2^\Z \}$ with the group law
$$ (n, x) \cdot (n',x') := (n+n', T^{n'}(x) + x').$$
The \emph{lamplighter group} $\Z \ltimes (\F_2^\Z)_0$ is the subgroup of $\Z \ltimes \F_2^\Z$ in which the second coordinate lies in $(\F_2^\Z)_0$.
\end{definition}

Our first main result, which we prove in Section \ref{lamp-sec}, classifies all sets of small doubling in the lamplighter and extended lamplighter groups up to constants.

\begin{theorem}[Freiman's theorem for the lamplighter group]\label{frei-lamp}  Let $(A, \Z \ltimes \F_2^\Z)$ be a multiplicative set of doubling constant at most $K$ for some $K \geq 1$.  Then $A$ is $O_K(1)$-controlled by a set $B$ which is of one of the following forms:
\begin{itemize}
\item[(Case 1)] $B = \{0\} \times V$ for some finite-dimensional subspace $V$ of $\F_2^\Z$.
\item[(Case 2)] $B = \{ (n,\phi(n)+v): n \in P, v \in V \}$, where $V$ is a $T^d$-invariant finite-dimensional subspace of $\F_2^\Z$ for some $d \geq 1$, $P \subset d \Z$ is a generalised arithmetic progression of rank $O_K(1)$, and $\phi: P \to \F_2^\Z$ is a function whose graph $n \mapsto (n,\phi(n))$ \emph{Freiman isomorphism modulo $V$} in the sense that
$$(n_1,\phi(n_1)) \cdot (n_2,\phi(n_2)) = (n_3,\phi(n_3)) \cdot (n_4,\phi(n_4)) \hbox{ mod } V$$ 
whenever $n_1,n_2,n_3,n_4 \in P$ are such that $n_1+n_2=n_3+n_4$. 
\end{itemize}
If $A \subset \Z \ltimes (\F_2^\Z)_0$, then one can take $V$ to be a finite-dimensional subspace of $(\F_2^\Z)_0$ in Case 1, and can take $V$ to be trivial in Case 2.
\end{theorem}

\begin{remark}  
In the converse direction, it is easy to see that the sets $B$ of the above form have small doubling, and thus any set controlled by them does also.  
This result strengthens an earlier result of Lindenstrauss \cite{linden}, who showed that the lamplighter group does not admit a F{\o}lner sequence of sets of uniformly small doubling constant.  Note that while $\F_2^\Z$ contains plenty of finite-dimensional invariant subspaces (e.g. the space of $L$-periodic sequences for some fixed period $L$), $(\F_2^\Z)_0$ contains no non-trivial invariant subspaces, which helps explain why the situation is simpler in that case.  In view of the above theorem, it seems of interest to classify all possible Freiman isomorphisms from generalised arithmetic progressions into $(\F_2^\Z)_0$ (or into $\F_2^\Z$ modulo an invariant subspace $V$), but this appears to be a somewhat difficult task.
\end{remark}

Now we turn to more general solvable groups.  We will also consider a special subclass of solvable groups, which we will call ``totally torsion-free solvable groups'':

\begin{definition}[Totally torsion-free solvable group]  Let $G = (G,\cdot)$ be a group.  The \emph{derived series} $G^{(0)} \geq G^{(1)} \geq \ldots$ of $G$ is formed by setting $G^{(0)} := G$ and $G^{(i+1)} := [G^{(i)},G^{(i)}]$ for $i=0,1,\ldots$; observe that these are normal subgroups of $G$.  A group is \emph{solvable} if $G^{(l+1)}$ is trivial for some $l$, known as the \emph{derived length} of $G$.  A group $G$ is \emph{torsion-free} if every non-identity element in $G$ has infinite order (thus if $a \in G$ and $a^n=1$ for some non-zero $n$, then $a=1$).  A solvable group $G$ is \emph{totally torsion-free} if $G/G^{(i)}$ is torsion-free for every $i$.
\end{definition}

\begin{example}  The $ax+b$ group of affine transformations $x \mapsto ax+b$ is a solvable group of derived length $2$ which is totally torsion free.  However, if one considers the subgroup $G$ of transformations $x \mapsto ax+b$ in which $a$ is a power of $3$, and $b$ is an integer divided by a power of $3$, then this group is still solvable of derived length $2$, and is torsion-free, but is no longer totally torsion-free.  Indeed, one easily verifies that $[G,G]$ consists of the translations $x \mapsto x+b$ in which $b$ is an \emph{even} integer divided by a power of $3$, and thus $G/G^{(1)}$ contains elements of order two (e.g. the image of the translation $x \mapsto x+1$).
\end{example}

Next, we need to generalise the notion of a coset progression to the virtually nilpotent setting.  

\begin{definition}[Coset nilprogression]\label{cosdef}  Let $G = (G,\cdot)$ be a group and $l \geq 0$ be an integer.  For each $1 \leq i \leq l$, let $r_i \geq 0$ be an integer, let $H_{i,0} = (H_{i,0},+)$ be a finite abelian group, and let $H_{i,j} := \{-N_{i,j},\ldots,N_{i,j}\}$ for $1 \leq j \leq r_i$ be intervals.  Suppose we have maps $\phi_{i,j}: H_{i,j} \to G$ for $0 \leq j \leq r_i$ obeying the following properties, where we abbreviate
$$ A_{i} := \phi_{i,0}(H_{i,0}) \cdot \ldots \cdot \phi_{i,r_i}(H_{i,r_i}); \quad A_{\geq i} := A_{l} \cdot \ldots \cdot A_{i}$$
\begin{itemize}
\item For all $1 \leq i \leq l$ and $0 \leq j \leq r_i$, we have 
\begin{equation}\label{ioi}
\phi_{i,j}(0) = 1.
\end{equation}
\item For any $1 \leq i \leq l$, $0 \leq j \leq r_i$, and $n, n' \in H_{i,j}$ with $n+n' \in H_{i,j}$, we have
\begin{equation}\label{okey}
 \phi_{i,j}(n) \phi_{i,j}(n') \in A_{\geq i+1} \cdot \phi_{i,j}(n+n').
 \end{equation}
\item For any $1 \leq i \leq l$, $0 \leq j,j' \leq r_i$, $n \in H_{i,j}$, and $n' \in H_{i,j'}$, we have
\begin{equation}\label{dokey}
 [\phi_{i,j}(n), \phi_{i,j'}(n')] \in A_{\geq i+1}.
 \end{equation}
\item For any $1 \leq i' < i \leq l$, $0 \leq j \leq r_i$, $0 \leq j' \leq r_{i'}$, $n \in H_{i,j}$, and $n' \in H_{i',j'}$, we have
\begin{equation}\label{soop}
 [\phi_{i',j'}(n'), \phi_{i,j}(n)] \in A_{\geq i+1} \cdot \phi_{i,0}(H_{i,0}) \cdot \ldots \cdot \phi_{i,\max(j-1,0)}(H_{i,\max(j-1,0)}).
\end{equation}
\end{itemize}
Then we refer to the multiplicative set 
$$(A_{\geq 1}, G) = ( \prod_{i=l}^1 \prod_{j=0}^{r_i} \phi_{i,j}(H_{i,j}), G)$$
as a \emph{coset nilprogression} of \emph{derived length} $l$, \emph{ranks} $r_1,\ldots,r_l$, and \emph{volume} $\prod_{i=1}^l \prod_{j=0}^{r_i} |H_{i,j}|$.  (If $l=0$, we adopt the convention that $A_{\ge 1} = \{1\}$.)

If the groups $H_{i,0}$ and maps $\phi_{i,0}$ are trivial, we refer to the coset nilprogression simply as a \emph{nilprogression}.
\end{definition}

\begin{examples} The only coset nilprogression of derived length $0$ is $\{1\}$.  A set $(A,G)$ is a coset nilprogression of derived length $1$ and rank $r$ if and only if it is a coset progression of rank $r$ in an abelian subgroup of $G$ (which we view additively).  The sets $B$ in Case 2 of Theorem \ref{frei-lamp} are coset nilprogressions of derived length $2$ and ranks $0, r$, where $r$ is the rank of the underlying generalised arithmetic progression $P$.  If $G$ is a $2$-step nilpotent group, $e_1, e_2 \in G$, and $N \geq 1$, then the set
$$ A := \{ [e_1,e_2]^{n_{12}} e_1^{n_1} e_2^{n_2}: n_1,n_2,n_{12} \in \Z; |n_1|, |n_2| \leq N; |n_{12}| \leq 100 N^2 \}$$
is a coset nilprogression of derived length $2$ and ranks $1,2$.   If $(A,G)$ is a coset nilprogression of derived length $l$ and ranks $r_1,\ldots,r_l$, and $0 \to H \to G' \to G \to 0$ is a finite abelian extension of $G$ (thus one has a projection homomorphism $\pi: G' \to G$ with kernel isomorphic to the finite abelian group $H$) then $(\pi^{-1}(A), G')$ is a coset nilprogression of derived length $l+1$ and ranks $r_1,\ldots,r_l,0$.
\end{examples}

\begin{example}
The following examples were contributed by Harald Helfgott. Let
$p$ be a large prime, let $F = F_p$ be the finite field of $p$ elements, and $G \subset SL_3(F)$
be the group of upper-triangular $3 \times 3$ matrices of determinant one and coefficients
in $F$; this is a solvable group of derived length $3$. Then the set
$$
A :=
\left\{ \begin{pmatrix} r^k & x & y \\ 0 & s^k & z \\ 0 & 0 & (rs)^{-k} \end{pmatrix}
: x,y,z \in F; -N \leq k \leq N \right\},$$
where $N \geq 1$ is an integer and $r,s \in F \backslash \{0\}$ are invertible elements, is a coset
nilprogression of derived length $3$ and ranks $1,0,0$. Indeed, we have $H_{1,0} = \{1\}$,
$H_{1,1}=\{-N,\ldots,N\}$, $H_{2,0} = F^2$, $H_{3,0} = F$ with
\begin{align*}
\phi_{1,1}(k) &:= \begin{pmatrix}
r^k & 0 & 0 \\
0 & s^k &0 \\
0 & 0 & (rs)^{-k}
\end{pmatrix}\\
\phi_{2,0}(x,y) &:= \begin{pmatrix}
1 &x & 0\\
0 & 1 & y \\
0 & 0 & 1
\end{pmatrix} \\
\phi_{3,0}(z) &= \begin{pmatrix}
1 & 0 & z \\
0 & 1&  0\\
0 & 0&  1
\end{pmatrix}.
\end{align*}
In a similar spirit, the set
$$
A :=
\left\{
\begin{pmatrix}
r^k & ms  & y\\
0 & r^k & z\\
0 & 0 & r^{-2k} 
\end{pmatrix}:
y,z \in \F; -M \leq m \leq M; -N \leq k \leq N
\right\}
$$
for integers $N,M \geq 1$, $r \in F \backslash \{0\}$, and $s \in F$ is a coset nilprogression of derived
lengths $3$ and ranks $1,1,0$ with $H_{1,0} = \{1\}$, $H_{1,1} = \{-N,\ldots,N\}$, $H_{2,0} = F$, $H_2 =\{-M,\ldots,M\}$, and $H_{3,0} = F$; we leave the construction of the $\phi_{i,j}$ to the interested reader.
\end{example}

\begin{remark} The property of being a coset nilprogression is preserved under Freiman homomorphisms\footnote{A \emph{Freiman homomorphism} $\phi: A \to A'$ of order $k$ from one multiplicative set $(A,G)$ to another $(A',G)$ is a map $\phi:A \to A'$ such that $\phi(a_1) \ldots \phi(a_k) = \phi(b_1) \ldots \phi(b_k)$ whenever $a_1,\ldots,a_k,b_1,\ldots,b_k \in A$ are such that $a_1 \ldots a_k = b_1 \ldots b_k$.} of sufficiently high order.  We will use this fact implicitly in Section \ref{kps}, when we lift a coset nilprogression from one group $G'$ to an extension $G$ of that group.  Also, given that doubling constants are preserved by Freiman isomorphisms, it is reassuring to know that coset nilprogressions are also, if one seeks to relate the former concept to the latter.
\end{remark}

\begin{remark}\label{tota} In the case that $G$ is totally torsion-free, we will be able to replace coset nilprogressions by nilprogressions throughout this paper.  However, the cosets appear to be necessary once this hypothesis is dropped. On the other hand, it is plausible to conjecture that one could be able to gather together all the finite groups in a coset nilprogression and quotient them out. More precisely, one could conjecture
that every coset nilprogression is controlled by a set $A$ containing a finite subgroup which is normal in $A$, and such that after quotienting out by this normal subgroup,
one obtains a standard nilprogression without any cosets. (Similar conjectures
have also been advanced by Harald Helfgott and by Elon Lindenstrauss (private
communications).) The results in this paper do not establish this fact, but it is
plausible that a refined analysis along these lines could establish this sort of result.
\end{remark}

Just as coset progressions are model examples of abelian sets of small doubling, coset nilprogressions turn out to be model examples of (solvable) small doubling:

\begin{lemma}[Polynomial growth of nilprogressions]\label{nilpoly}  Let $(A,G)$ be a coset nilprogression of derived length $l$ and ranks $r_1,\ldots,r_l$.  Then we have $|A^{\pm n}| \ll_{l,r_1,\ldots,r_l} n^{O_{l,r_1,\ldots,r_l}(1)} |A|$ for all $n \geq 1$.
\end{lemma}

We prove this lemma in Section \ref{cosnil}.  The main result of this paper is the following converse to the above proposition in the solvable case:

\begin{theorem}[Freiman's theorem for solvable groups]\label{frei-solv}  Let $(A,G)$ be a multiplicative set in a solvable group $G$ of derived length at most $l$ and of doubling constant at most $K$ for some $K,l \geq 1$.  Then there exists a coset nilprogression $(A',G)$ of derived length $l$ and all ranks $O_{K,l}(1)$ and volume $\sim_{K,l} |A|$ which $O_{K,l}(1)$-controls $(A,G)$.  Furthermore we have $A' \subset A^{\pm O_{K,l}(1)}$.

If $G$ is totally torsion-free, we can ensure that $A'$ is a nilprogression rather than a coset nilprogression.
\end{theorem}

\begin{remark} Our proof of Theorem \ref{frei-solv} is completely effective, and in principle one could write down explicit values for the bounds $O_{K,l}(1)$ given above. However, due to our reliance on tools such as the Szemer\'edi regularity lemma, the bounds are incredibly poor (roughly speaking, they are an $l$-fold iterated tower exponential in $K$), and so we do not attempt to quantify them here.\footnote{Note added in proof: thanks to recent improvements in additive combinatorics technology, in  
particular the lemmas in \cite{sanders-2}, \cite{croot}, one could avoid use of the regularity lemma here and obtain
bounds which are ``merely'' of iterated exponential type rather than iterated tower-exponential.
However, there are still exponential losses in these lemmas and so even if one optimistically
assumes a ``polynomial Freiman-Ruzsa conjecture'', these methods do not seem to lead currently
to polynomial bounds. On the other hand, in the case of solvable linear groups over $\C$, the
recent results in \cite{bg2} only have exponential losses in the constants, while for simple linear groups
the best results are polynomial in nature (see \cite{bgt}, \cite{gill}, \cite{pyber}, \cite{varju} for the most recent results in this highly active area), and so it is conceivable that some refinement of these techniques could
eventually lead to polynomially quantitative results, conditionally on a polynomial Freiman-Ruzsa
type conjecture.}
\end{remark}

\begin{remark} Theorem \ref{frei-solv} generalises Theorem \ref{gr2} (which is basically the $l=1$ case). It is not a completely satisfactory description of the sets of small doubling in a solvable group, because we do not have a completely explicit classification of coset nilprogressions; this problem is analogous to the problem of classifying the Freiman isomorphisms appearing in Theorem \ref{frei-lamp} (and broadly analogous to the infamous extension problem for groups, in particular having a similar ``cohomological'' flavour), and will not be pursued here; see however \cite{bg} for some relevant recent progress in this direction. 
\end{remark}

In spite of the above deficiencies, we can use Theorem \ref{frei-solv}, together with some basic properties of coset nilprogressions, to establish some non-trivial applications.  For instance, we have the following result, proven in Section \ref{cosnil}:

\begin{lemma}[Virtual nilpotency of nilprogressions]\label{virtnil} Let $(A,G)$ be a coset nilprogression of derived length $l$ and ranks $r_1,\ldots,r_l$.  Then the group $\langle A \rangle$ generated by $A$ contains a subgroup $N$ of index $O_{l,r_1,\ldots,r_l}(|A|^{O_{l,r_1,\ldots,r_l}(1)})$ which is nilpotent of step $O_{l,r_1,\ldots,r_l}(1)$ and is generated by $O_{l,r_1,\ldots,r_l}(1)$ generators.

If $(A,G)$ is a nilprogression rather than a coset nilprogression, then $\langle A \rangle$ is itself nilpotent of step $O_{l,r_1,\ldots,r_l}(1)$ and is generated by $O_{l,r_1,\ldots,r_l}(1)$ generators.
\end{lemma}

\begin{remark} Lemma \ref{virtnil} is not terribly surprising, in view of Lemma \ref{nilpoly}, and the close connection between polynomial growth and virtual nilpotency (see \cite{milnor}, \cite{wolf}, \cite{gromov}).   As mentioned in Remark \ref{tota}, one could conjecture a stronger result
than Lemma \ref{virtnil}, namely that $A$ is $O_{l,r_1,\ldots,r_l}(1)$-controlled by a set which becomes
a nilprogression of rank and step $O_{l,r_1,\ldots,r_l}(1)$ after quotienting out by a normal
subgroup.
\end{remark}

Combining Lemmas \ref{nilpoly}, \ref{virtnil} with Theorem \ref{frei-solv}, we conclude

\begin{corollary}[Freiman's lemma for solvable groups]\label{Freilem}  Let $(A,G)$ be a multiplicative set in a solvable group of derived length at most $l$, and a doubling constant of at most $K$.  Then there exists subgroups $N \leq H \leq G$, with $N$ nilpotent of step $O_{K,l}(1)$ and generated by $O_{K,l}(1)$ generators with index $O_{K,l}( |A|^{O_{K,l}(1)} )$ in $H$, such that $A$ is $O_{K,l}(1)$-controlled by a subset $A'$ of $H$.  Furthermore we have $A' \subset A^{\pm O_{K,l}(1)}$ and $|(A')^{\pm n}| \ll_{K,l} n^{O_{K,l}(1)} |A|$ for all $n \geq 1$.

If $G$ is totally torsion-free, one can take $H=N$.
\end{corollary}

This theorem asserts, roughly speaking, that any set of small doubling in a solvable group is controlled by a subset of a virtually nilpotent group of bounded rank.  It is analogous to Freiman's lemma\cite{freiman} that an additive set in a torsion-free abelian group of small doubling is contained in a subgroup generated by a bounded number of generators. Again, one could conjecture that stronger results hold.

In a similar spirit, we have the following result, proven in Section \ref{applied}:

\begin{theorem}[Milnor-Wolf type theorem]\label{mwt}  Let $G$ be a solvable group of derived length at most $l$, let $S$ be a finite set of generators for $G$.  For any $r > 0$, let $B_S(r) := S^{\pm \lfloor r\rfloor}$ be the ball of radius $r$.  Suppose that $|B_S(R)| \leq R^d$ for some $d > 0$ and $R > 1$. If $R$ is sufficiently large depending on $l$, $d$, then $G$ is virtually nilpotent.  More specifically, $G$ contains a nilpotent subgroup $N$ of step $O_{l,d}(1)$ and index $O(R^{O_{l,d}(1)})$.  Furthermore, we have the bound
$$ |B_S(r)| \ll_{l,d} r^{O_{l,d}(1)}$$
for all $r \geq R$.
\end{theorem}

This strengthens a classical theorem of Milnor\cite{milnor} and Wolf\cite{wolf}, which asserts that any solvable group of polynomial growth must be virtually nilpotent; the above result gives a more quantitative version of that statement, which only requires polynomial growth at a single (large) scale, and gives some control on the nature of the virtual nilpotency.  As one consequence of this stronger statement, we see that balls in solvable groups cannot transition from polynomial growth to non-polynomial growth at large scales\footnote{Of course, the reverse transition from non-polynomial growth to polynomial growth is easy to attain.  Consider for instance a large finite group generated by a set whose Cayley graph has large girth; balls in this group grow exponentially at first, but are ultimately constant.}.  In view of Gromov's celebrated theorem\cite{gromov} that extends the Milnor-Wolf theorem to arbitrary finitely generated groups of polynomial growth, it seems of interest to ask whether there is an analogue of Theorem \ref{mwt} for groups that are not necessarily solvable\footnote{Note added in proof: this question has now solved affirmatively; see \cite{shalom}.}.  (It may also be that Corollary \ref{Freilem} can similarly be generalised to non-solvable groups.)

Finally, we mention two auxiliary results, which were established in this paper in order to prove the above theorems, but which may have some independent interest.  In Proposition \ref{sarkozy} we establish a variant of S\'ark\H{o}zy's theorem, which roughly speaking will state that if an additive set is densely contained inside a coset progression, then some iterated sumset of that set will itself contain a large coset progression.  In Proposition \ref{bz-prop2} we show that every multiplicative set $A$ of bounded doubling controls a large set whose iterated sumset is mostly contained in $A \cdot A^{-1}$; the point here is that the number of iterations can be taken to be rather large\footnote{Note added in proof: this result has since been significantly strengthened, see \cite{croot}, \cite{sanders-2}.}.

\subsection{Connections with other non-abelian Freiman theorems}

The work here complements some recent work in \cite{sanders}, \cite{bg}, \cite{bg2}, \cite{fisher}.  For instance, the Freiman-type theorems here essentially reduce the study of sets of small doubling in the (virtually) solvable case to the (virtually) nilpotent case.  The latter case is then studied further in \cite{bg}, \cite{fisher}.  For instance, in \cite{fisher} it was shown (by many applications of the Baker-Campbell-Hausdorff formula) that if $A$ was a multiplicative set of small doubling inside a simply connected nilpotent Lie group $G$ of bounded step, then the logarithm of $A$ (which can be viewed as an additive set in the Lie algebra of $G$) had small (additive) doubling and could thus be controlled by the additive Freiman theorem.  This analysis was then taken further in \cite{bg}, in which it was shown that if $A$ was a multiplicative set of small doubling inside a torsion-free nilpotent group of bounded step, then $A$ was controlled by a \emph{nilpotent progression}, which is a special case of the coset nilprogressions studied here, but with significantly more structure.  More precisely: the torsion components $H_{i,0}$ are trivial, the $\phi_{i,j}$ are homomorphisms $\phi_{i,j}(n) := e_{i,j}^n$, the generators $e_{i,j}$ are formed as commutators of a list $e_1,\ldots,e_k$ of base generators (ordered in a standard fashion), and the lengths $N_{i,j}$ are products of base lengths $N_1,\ldots,N_k$, where the exponents correspond to the weight of each base generator in the commutator $e_{i,j}$; see \cite{bg} for further discussion.  These results combine well with our results in the case when $G$ is totally torsion-free, in which case they assert that any multiplicative set in $G$ of bounded doubling is controlled by a nilpotent progression.  The situation seems to be more complicated however in the presence of torsion, as this seems to generate some non-trivial ``cohomology''.

In a somewhat different direction, the work of \cite{sanders} relates sets of polynomial growth (thus $|A^n| \leq n^d |A|$ for all large $n$) with balls in a translation-invariant pseudo-metric, in the case when $G$ is a \emph{monomial group} (a concept which overlaps with, but is not entirely equivalent to, that of a solvable group).  Note that the nilpotent progressions studied in \cite{bg} are essentially balls of this type.  This provides part of an alternate path to non-abelian Freiman theorems, although the precise connection between small doubling and polynomial growth is not well understood yet, except in the abelian (and nilpotent) settings, where one can pass from one to the other fairly easily.

The results of \cite{bg} have been extended very recently to linear solvable groups in \cite{bg2} (private communication).

\subsection{Notation}\label{notation-sec}

We use the asymptotic notation $X \ll Y$, $Y \gg X$, or $X = O(Y)$ to denote the estimate $|X| \leq CY$ for some absolute constant $C$, and $X \sim Y$ for $X \ll Y \ll X$.  We will often need the implied constant $C$ to depend on one or more parameters, in which case we shall indicate this by subscripts unless otherwise noted, thus for instance $X \ll_K Y$ means that $|X| \leq C_K Y$ for some quantity $C_K$ depending only on $K$.

We combine the above asymptotic notation with the sumset and product set notation, e.g. $O_K(1) A$ denotes an iterated sumset $kA = A + \ldots + A$ for some $k = O_K(1)$, etc.

Because we wish to reserve $A^k$ for the $k$-fold product set $A \cdot \ldots \cdot A$ of a set $A$, we will use the notation $A^{\otimes k}$ to denote the $k$-fold Cartesian product $A \times \ldots \times A$.

Given a subset $A$ of a group $G$, we use $\langle A \rangle$ to denote the group generated by $A$.

Let $G = (G,\cdot)$ be a group.  The commutator $[g,h]$ of two group elements $g,h \in G$ is defined as $[g,h] := ghg^{-1}h^{-1}$.  The commutator $[H,K]$ of two subgroups $H,K \leq G$ is defined as the group generated by the commutators $[h,k]$ for $h \in H, k \in K$.  The \emph{lower central series} $G = G_1 \geq G_2 \geq \ldots$ of a group $G$ is defined recursively by $G_1 := G$ and $G_{i+1} := [G,G_i]$; a group $G$ is \emph{nilpotent} of step at most $s$ if $G_{s+1}$ is trivial.  A group is \emph{virtually nilpotent} if it has a nilpotent subgroup of finite index.  

\subsection{Acknowledgements}

The author is indebted to Ben Green and Tom Sanders for valuable discussions, particularly in relation to the recent preprints \cite{bg}, \cite{sanders}.  The author also thanks the referee for suggestions, and David Speyer for corrections.  The author is supported by a grant from the MacArthur Foundation, and by NSF grant DMS-0649473. 

\section{The lamplighter group}\label{lamp-sec}

We now prove Theorem \ref{frei-lamp}, which is a much simpler result than Theorem \ref{frei-solv} but already contains some of the key ideas, most notably a focus on analysing group actions of a multiplicative set $A$ on an additive set $E$ which have small ``doubling constant'' in some sense.

In view of Lemma \ref{small}(i) (and the transitivity property mentioned in Remark \ref{control}), we may assume without loss of generality that $A$ is a $K$-approximate group (see Definition \ref{approx}), rather than merely a set of small doubling.  We allow all implied constants to depend on $K$, and henceforth omit the dependence of $K$ in the subscripts.

Let $\pi: \Z \ltimes \F_2^\Z \to \Z$ be the canonical projection, thus $\ker(\pi) \equiv \F_2^\Z$ is abelian.  By Lemma \ref{proj}(i), $\pi(A) \subset \Z$ is also a $K$-approximate group.  By Theorem \ref{gr2} (or from the results in \cite{ruzsa} or \cite{chang}), $\pi(A)$ is controlled by a generalised arithmetic progression $P \subset \pi(A^4)$ of rank $O(1)$.  In particular, $|\pi(A)| \sim |P|$.

Let $C \geq 1$ be a large constant depending on $K$ to be chosen later.  By dropping all small dimensions of $P$, we may assume that all dimensions $N_1,\ldots,N_r$ of $P$ are at least $C$.  We may also of course assume that the generators $v_1,\ldots,v_r$ of $P$ are non-zero.

By Lemma \ref{proj}(iii), $(\ker(\pi) \cap A^{100})^3 \subset \F_2^\Z$ is a $O(1)$-approximate group of cardinality $\sim |A|/|P|$; by Theorem \ref{gr} (or Ruzsa's Freiman theorem \cite{ruzsa-group}) we can thus find a subspace $W$ of $\F_2^\Z$ containing $(\ker(\pi) \cap A^{100})^3$ of cardinality $\sim |A|/|P|$.  

Consider the set $E := \ker(\pi) \cap A^{20}$.  Clearly $E \subset W$. From Lemma \ref{proj}(ii), we have $|E| \sim |W|$.  Also, if $v \in P$, we see on conjugating $E$ by an element of $A^4 \cap \pi^{-1}(\{v\})$ that $T^{v}(E) \subset \ker(\pi) \cap A^{28} \subset V$.  Thus for $0 \leq n \leq C$, the linear space 
$$W_n := \operatorname{span}( \bigcup_{-n \leq a_1,\ldots,a_r \leq n} T^{a_1 v_1 + \ldots + a_r v_r}(E) )$$ 
is a subspace of $V$ of cardinality $\sim |W|$.  These spaces are also non-decreasing in $n$.  By the pigeonhole principle (and the fact that finite-dimensional subspaces of $W$ have cardinality equal to a power of $2$), we thus conclude (if $C$ is large enough) that there exists $0 \leq n < C$ such that $W_n = W_{n+1}$. Setting $V := W_n$, we conclude that $|V| \sim |W| \sim |A|/|P|$, that $V$ contains $\ker(\pi) \cap A^{20}$, and $V$ is invariant with respect to $T^{v_i}$ for $i=1,\ldots,r$.

There are two cases.  If $r=0$, then $|V| \sim |A|$, and from Lemma \ref{small}(iii) we see that $V$ $O(1)$-controls $A^2$ and hence $O(1)$-controls $A$, and we are in Case 1.  Now suppose instead that $r > 0$.  Setting $d \geq 1$ to be the greatest common divisor of the $v_1,\ldots,v_r$, we conclude that $V$ is $T^d$-invariant.  Also we have $P \subset d\Z$.

Now consider the symmetric set $B := V \cdot (\pi^{-1}(P) \cap A^4)$.  On the one hand, from Lemma \ref{proj}(ii) we have $\pi(B) = P$.  
On the other hand, since $V$ contains $\ker(\pi) \cap A^{20}$ and is invariant under any shift arising from $P$, we see that $\ker(\pi) \cap B^4 = V$.  
We conclude that $B$ takes the form 
$$ B = \bigcup_{n \in P} (n,\phi(n)) + V$$
for some function $\phi: P \to \F_2^\Z$ whose graph $n \mapsto (n,\phi(n))$ is a Freiman isomorphism modulo $V$ in the sense that $(n_1,\phi(n_1)) \cdot (n_2,\phi(n_2)) = (n_3,\phi(n_3)) \cdot (n_4,\phi(n_4)) \hbox{ mod } V$ whenever $n_1,n_2,n_3,n_4 \in P$ are such that $n_1+n_2=n_3+n_4$.    This makes $B$ into a $O(1)$-approximate group of size $|P| |V| \sim |A|$.  We thus see that $B$ $O(1)$-controls $\pi^{-1}(P) \cap A^4$, which (by Lemma \ref{small}(iii)) $O(1)$-controls $A^4$, which $O(1)$-controls $A$, and we are in Case 2 as desired.

Finally, if $A$ was initially in $\Z \ltimes (\F_2^\Z)_0$ instead of $\Z \ltimes \F_2^\Z$, then we can repeat the above arguments with $\F_2^\Z$ replaced by $(\F_2^\Z)_0$ throughout.  Since $(\F_2^\Z)_0$ does not contain any non-trivial finite-dimensional $T^d$-invariant subspaces for any $d \geq 1$, we obtain the final claims in Theorem \ref{frei-lamp}.

\section{Properties of coset nilprogressions}\label{cosnil}

We now prove a number of basic facts about coset nilprogressions stated in the introduction, including Lemma \ref{nilpoly} and Lemma \ref{virtnil} from the introduction.  All the properties here are rather trivially true for coset progressions, and the generalisation to coset nilprogressions is relatively routine, requiring mostly a certain amount of ``nilpotent algebra'' bookkeeping.

It is convenient to introduce the following notion.

\begin{definition}[Dilation]  Let 
$$ A = \prod_{i=l}^1 \prod_{j=0}^{r_i} \phi_{i,j}(H_{i,j})$$
be a coset nilprogression of derived length $l$ and ranks $r_1,\ldots,r_l$.  For any tuple $M = (M_{i,j})_{1 \leq i \leq l; 1 \leq j \leq r_i}$ of nonnegative integers, we define the dilate $A^M$ of $A$ to be the set
$$ A^M = \prod_{i=l}^1 \prod_{j=0}^{r_i} \phi_{i,j}(H_{i,j})^{M_{i,j}}$$
with the convention that $M_{i,0}$ is always equal to $1$.
\end{definition}

We have the following basic estimates:

\begin{lemma}[Multiplication laws]\label{mult}  Let $A = (A,G)$ be a coset nilprogression of derived length $l$ and ranks $r_1,\ldots,r_l$, and let $M = (M_{i,j})_{1 \leq i \leq l; 1 \leq j \leq r_i}$, $M' = (M'_{i,j})_{1 \leq i \leq l; 1 \leq j \leq r_i}$ be tuples of nonnegative integers.  Then we have
\begin{equation}\label{amm}
 A^M \cdot A^{M'} \subset A^{\tilde M + \tilde M'}
 \end{equation}
where $\tilde M = (M_{i,j})_{1 \leq i \leq l; 1 \leq j \leq r_i}$ is of the form
$$ \tilde M_{i,j} = M_{i,j} + O_{l,r_1,\ldots,r_l}( \sum_{j < j' \leq r_i} M_{i,j'} + \sum_{1 \leq i' < i} \sum_{0 \leq j' \leq r_{i'}} M_{i',j'} )^{O_{l,r_1,\ldots,r_l}(1)}$$
(again adopting the convention $M_{i,0}=1$) and similarly for $\tilde M'$.  In a similar spirit, we have
\begin{equation}\label{ami}
 (A^M)^{-1} \subset A^{\tilde M}.
\end{equation}
\end{lemma}

\begin{proof} We begin with \eqref{amm}.  The claim is vacuously true for $l=0$, so suppose inductively\footnote{The reader may wish to work out some small cases by hand, e.g. $l=2$, $r_1=r_2=1$, to get a sense of what is going on here.  The arguments here are similar to those that establish that a group generated by a set $S$ is nilpotent of step $s$ if all $s+1$-fold iterated commutators of the generators vanish; there is a very similar ``nilpotent'' or ``upper triangular'' nature to the algebra involved here.} that $l \geq 1$ and the claim has already been proven for $l-1$.

Next, we induct on $r_1$ (keeping $l$ fixed).  If $r_1=0$, then we can write $A^M = (A_{\geq 2})^{\overline{M}} \cdot \phi_{1,0}(H_{1,0})$, where $A_{\geq 2}$ is a coset progression of derived length $l-1$ and ranks $r_2,\ldots,r_l$, and $\overline{M}$ is the truncation of $M$ formed by omitting the (empty) $i=1$ component of $M$ and then relabeling.  Thus
$$ A^M \cdot A^{M'} = (A_{\geq 2})^{\overline{M}} \cdot \phi_{1,0}(H_{1,0}) \cdot (A_{\geq 2})^{\overline{M'}} \cdot \phi_{1,0}(H_{1,0}).$$
Let $a \in \phi_{1,0}(H_{1,0})$, then from \eqref{soop} we have
$$ a \cdot \phi_{i,j}(H_{i,j}) \cdot a^{-1} \subset A_{\geq i+1} \cdot \phi_{i,0}(H_{i,0}) \cdot \ldots \cdot \phi_{i,j}(H_{i,j})$$
for all $2 \leq i \leq l$ and $1 \leq j \leq r_i$, and similarly
$$ a \cdot \phi_{i,j}(H_{i,0}) \cdot a^{-1} \subset A_{\geq i+1} \cdot \phi_{i,0}(H_{i,0}) \cdot \phi_{i,0}(H_{i,0}).$$
On the other hand, from \eqref{okey} one has
\begin{equation}\label{io}
 \phi_{i,0}(H_{i,0}) \cdot \phi_{i,0}(H_{i,0}) \subset A_{\geq i+1} \cdot \phi_{i,0}(H_{i,0}).
 \end{equation}
By repeated application of the induction hypothesis, we thus conclude that
$$ a \cdot (A_{\geq 2})^{\overline{M'}} \cdot a^{-1} \subset (A_{\geq 2})^{\overline{\tilde M'}} $$
if the constants used in the definition of $\tilde M'$ are chosen appropriately.  We conclude that
$$ \phi_{i,0}(H_{i,0}) \cdot (A_{\geq 2})^{\overline{M'}} \subset (A_{\geq 2})^{\overline{\tilde M'}} \cdot \phi_{i,0}(H_{i,0})$$
and hence
$$ A^M \cdot A^{M'} = (A_{\geq 2})^{\overline{M}} \cdot (A_{\geq 2})^{\overline{M'}} \cdot \phi_{1,0}(H_{1,0}) \cdot \phi_{1,0}(H_{1,0}).$$
Using \eqref{io} and the induction hypothesis again, we obtain the claim (after modifying the definition of $M'$ slightly).

The case $r_1 > 0$ can be recovered from the inductively preceding case $r_1-1$ by a very similar argument which we omit (the only real differences are that one uses \eqref{dokey} as well as \eqref{soop}, and we no longer have \eqref{io}, and so must allow the factors of $\phi_{i,r_i}(H_{i,r_i})$ to accumulate).

To prove \eqref{ami}, observe from \eqref{okey} (setting $n' := -n$) that
$$ \phi_{i,j}(H_{i,j})^{-1} \subset A_{\geq i+1}^{-1} \cdot \phi_{i,j}(H_{i,j})$$
for all $1 \leq i \leq l$ and $0 \leq j \leq r_i$.  Iterating this repeatedly and then using \eqref{ami} we obtain the claim.
\end{proof}

From the above lemma and a routine induction argument we see that
\begin{equation}\label{sock}
A^{\pm k} \subset A^{(k^{c_{i,j}})_{1 \leq i \leq l; 1 \leq j \leq r_i}}
\end{equation}
for all $k \geq 1$ and for suitable choices of constants $c_{i,j}$ depending only on $l,r_1,\ldots,r_l$ (one wants $c_{i,j}$ to be decreasing in $j$, and very rapidly increasing in $i$).

Now we cover $A^M$ by $A$.

\begin{lemma}[Covering properties of coset nilprogressions]\label{control-o} Let $(A,G)$ be a coset nilprogression of derived length $l$ and ranks $r_1,\ldots,r_l$, and let $M = (M_{i,j})_{1 \leq i \leq l, 1 \leq j \leq r_i}$ be a tuple of non-negative integers.  Then $A^M$ is $O_{l,r_1,\ldots,r_l}( 1 + \sum_{i=1}^l \sum_{j=1}^{r_i} M_{i,j} )^{O_{l,r_1,\ldots,r_l}(1)}$-controlled by $A$.
\end{lemma}

\begin{proof}  We allow implied constants to depend on $l,r_1,\ldots,r_l$, and abbreviate $\sum_{i=1}^l \sum_{j=1}^{r_i} M_{i,j}$ as $|M|$.  As in Lemma \ref{mult}, we induct on $l$, the $l=0$ case being trivial.  In view of \eqref{ami}, it suffices to show that $A^M \subset X \cdot A$ for some set $X$ of size $|X| \ll (1+|M|)^{O(1)}$.

First consider the case $r_1=0$.  Then $A^M = A_{\geq 2}^{\overline{M}} \cdot \phi_{1,0}(H_{1,0})$, and the claim follows from the induction hypothesis.  Now suppose inductively that $r_1 > 0$, and the claim has already been proven for $r_1-1$.

Write $m_{1,r_1} := \lfloor M_{1,r_1}/2 \rfloor$.  By repeated use of \eqref{okey} (and \eqref{ioi}), we can express any element of $\phi_{1,r_1}(H_{1,r_1})^{M_{1,r_1}}$ as a word consisting of $O(|M|)$ elements of $A_{\geq 2}$ or $A_{\geq 2}^{-1}$, $O(|M|)$ instances of $\phi_{1,r_1}(m_{1,r_1})$ or $\phi_{1,r_1}(m_{1,r_1})^{-1}$, and a single instance of an element of $\phi_{1,r_1}(H_{1,r_1})$.  Permuting these terms using \eqref{dokey}, \eqref{soop}, and Lemma \ref{mult}, one eventually arrives at
$$ \phi_{1,r_1}(H_{1,r_1})^{M_{1,r_1}} \subset \{ \phi_{1,r_1}(m_{1,r_1})^n: n = O(|M|) \} \cdot A_{\geq 2}^{(O(|M|)^{O(1)})_{2 \leq i \leq l; 1 \leq j \leq r_i} \cdot \phi_{1,r_1}(H_{1,r_1})}.$$
The left-hand side here is the final term in $A^M$.  Inserting the above inclusion and using \eqref{dokey}, \eqref{soop}, and Lemma \ref{mult} repeatedly, one eventually obtains
$$ A^M \subset \{ \phi_{1,r_1}(m_{1,r_1})^n: n = O(|M|) \} \cdot (A')^{M'} \cdot \phi_{1,r_1}(H_{1,r_1})$$
where $A'$ is the coset nilprogression formed from $A$ by decrementing $r_1$ by $1$ (thus dropping $H_{1,r_1}$ and $\phi_{1,r_1}$), and $M'$ is some tuple with all exponents $O(1+|M|)^{O(1)}$.  Applying the induction hypothesis we see that $(A')^{M'} \subset X' \cdot A'$ for some $|X'| \ll (1+|M|)^{O(1)}$.  Since $A = A' \cdot \phi_{1,r_1}(H_{1,r_1})$, we obtain the claim.
\end{proof}

Combining \eqref{sock} and Lemma \ref{control-o}, we obtain

\begin{corollary}[Coset nilprogressions control their sumsets]\label{control-cor} Let $(A,G)$ be a coset nilprogression of derived length $l$ and ranks $r_1,\ldots,r_l$.  Then for any $k \geq 1$, $A^{\pm k}$ is $O_{l,r_1,\ldots,r_l}(k^{O_{l,r_1,\ldots,r_l}})$-controlled by $A$.
\end{corollary}

\begin{remark} This corollary implies that coset nilprogressions behave very much like $O_{l,r_1,\ldots,r_l}(1)$-approximate groups, though they are not quite so, mainly because coset nilprogressions are not quite symmetric (cf. \eqref{ami}).
\end{remark}

It is clear that Lemma \ref{nilpoly} is an immediate consequence of Corollary \ref{control-cor}.  Now we turn to Lemma \ref{virtnil}.  We begin with a purely algebraic lemma.

\begin{lemma}[Finite extensions of nilpotent groups are virtually nilpotent]\label{nil-lem}  Let $G$ be a nilpotent group of step $s$ generated by $k$ generators, and let $G'$ be a finite extension of $G$, thus one has the short exact sequence
$$ 0 \to H \to G' \to G \to 0$$
for some finite abelian group $H$.  Then $G'$ contains a subgroup $N$ of index $O_{s,k}( |H| )^{O_{s,k}(1)}$ which is nilpotent of step $s$ generated by $O_{s,k}(1)$ generators.
\end{lemma}

\begin{proof}  We allow all implied constants to depend on $s,k$.
We may view $H$ as a normal subgroup of $G'$, with $G \equiv G'/H$.  Let $\pi: G' \to G$ be the projection map, and $e_1,\ldots,e_k$ be a set of generators for $G$. We lift each $e_j$ up to an element $e'_j \in \pi^{-1}(\{e_j\})$ of $G'$.

Since $H$ is abelian, we see that the conjugation action of $G'$ on $H$ descends to an action $\rho: G \to \operatorname{Aut}(H)$ of $G$ on $H$.  
Let us first consider the special case when $\rho$ is trivial, or equivalently that $H$ is a central subgroup of $G'$; since $G$ is nilpotent of step $s$, this makes $G'$ nilpotent of step $s+1$, with the $(s+1)^{th}$ commutator group $G'_{s+1}$ in the lower central series contained in $H$.  
Then we set $N$ to be
the group generated by $\{ g^{M! |G'_{s+1}|}: g \in G \}$ for some large integer $M = O(1)$. On
the one hand, $\pi(N)$ is clearly a normal subgroup of $G$, and $G/\pi(N)$ is
an $M! |G'_{s+1}|$-torsion nilpotent group of step $s$ generated by $k$ generators, and thus
$\pi(N)$ has index $O_M(|H|^{O(1)})$ in $G$; since $G'$ is an extension of $G$ by $H$, we conclude
that $N$ has index $O_M(|H|^{O(1)})$ in $G'$. On the other hand, the $(s+1)^{th}$ commutator
group of $N$ is generated by $s + 1$-fold commutators of $g^{M!|G_{s+1}|}$ for $g \in G'$. By
repeated use of the commutator identity
\begin{equation}\label{comm}
[ xy, z ] = [x,[y,z]] [y,z] [x,z]
\end{equation}
and exploiting the nilpotency of $G'$, we see that we can express any such $s+1$-fold commutator 
in the form 
$$ f_1^{P_1(n)} \ldots f_J^{P_J(n)}$$
for some $f_1,\ldots,f_J \in G'_{s+1}$ and some polynomials $P_J(n)$ of degree $O(1)$ whose coefficients have numerator and denominator $O(1)$, and with vanishing constant term.  Setting $n := M! |G'_{s+1}|$ for $M=O(1)$ large enough and using Lagrange's theorem $f^{|G'_{s+1}|} = 1$ for all $f \in G'_{s+1}$, we see that these commutators vanish, and so $N$ is nilpotent of step $s$.  This concludes the claim when $\rho$ is trivial.

To handle the general case when $\rho$ is non-trivial, we need a key claim: given any $e \in G$ and $h \in H$, there exists an integer $1 \leq n \ll |H|^{O(1)}$ depending on $e,h$ such that the normal subgroup $N_{e^n}$ of $G$ generated by $e^n$ fixes the element $h$ (with respect to the action $\rho$).  To see this, observe from the nilpotency of $G$ that $N_{e^n}$ is generated by $O(1)$ iterated commutators, each of the form 
\begin{equation}\label{enr}
[\ldots[e^n, e_{i_1}], \ldots, e_{i_r}]
\end{equation}
for some $0 \leq r < s$ and $1 \leq i_1,\ldots,i_r \leq k$.  By repeated use of the commutator identity
$$ [ xy, z ] = [x,[y,z]] [y,z] [x,z]$$
and exploiting the nilpotency of $G$, we see that we can express \eqref{enr} in the form
$$ f_1^{P_1(n)} \ldots f_J^{P_J(n)}$$
where $J=O(1)$, $f_1,\ldots,f_J$ depend on $e, e_{i_1}, \ldots, e_{i_r}$ but are independent of $n$, and the $P_j(n)$ are polynomials of degree $O(1)$ whose coefficients are rationals of numerator and denominator $O(1)$.  Furthermore, since \eqref{enr} vanishes when $n=0$, we can ensure that $P_j(0)=0$ for all $j$.

For each $1 \leq j \leq J$, we see from the pigeonhole principle applied to the sequence $h, \rho(f_j) h, \rho(f_j^2) h, \ldots$ that there exists an integer $1 \leq n_j \leq |H|$ such that $\rho(f^{n_j}) h = h$.  If we set $n := M! n_1 \ldots n_J$ for some sufficiently large integer $M=O(1)$ (so that $n = O( |H|^{O(1)})$), we see that $P_j(n)$ is a multiple of $n_j$ for each $1 \leq j \leq J$, and thus all of the commutators \eqref{enr} fix $h$.  Since these commutators generate $N_{e^n}$, the claim follows.

Next, we lift each $e_j$ arbitrarily up to some element $e'_j \in \pi^{-1}(\{e_j\})$.  Since the $s+1$-fold iterated commutators of $e_1,\ldots,e_k$ vanish, we have
\begin{equation}\label{eis}
 [\ldots[e'_{i_1},e'_{i_2}],\ldots,e'_{i_{s+1}}] =: h_{i_1,\ldots,i_{s+1}} \in H
 \end{equation}
for all $1 \leq i_1,\ldots,i_{s+1} \leq k$ and some $h_{i_1,\ldots,i_{s+1}} \in H$.

By the claim just proven (and taking least common multiples), we can find an integer $1 \leq n \ll |H|^{O(1)}$ such that $N_{e_i^n}$ fixes all of the $h_{i_1,\ldots,i_{s+1}}$, or equivalently that $h_{i_1,\ldots,i_{s+1}}$ is centralised by $\pi^{-1}(N_{e_i^n})$.  Since the $\pi^{-1}(N_{e_i^n})$ are normal, they also centralise $h_{i_1,\ldots,i_{s+1}}$.  If we let $H' \leq H$ be the normal subgroup of $G'$ generated by $h_{i_1,\ldots,i_{s+1}}$ and all its conjugates, we thus see that $H'$ is centralised by each of the $\pi^{-1}(N_{e_i^n})$.  

If we now let $N$ be the subgroup of $G'$ generated by $(e'_1)^n,\ldots,(e'_k)^n$, then we see that each of the $(e'_i)^n$ centralise $H'$, and so $N$ centralises $H$.  The claim now follows from the case of trivial $\rho$ discussed at the beginning of the proof.
\end{proof}

\begin{remark} If $H$ is not abelian, then one can obtain a weaker version of Lemma \ref{nil-lem} in which the index of $N$ is only shown to be $O( |H| )^{O( \log |H| )}$ rather than $O( |H| )^{O(1)}$.  We sketch the proof as follows.  The conjugation action of $G'$ on $H$ yields a homomorphism $\phi$ from $G'$ to $\operatorname{Aut}(H)$.  On the other hand, a greedy argument shows that $H$ is generated by $O( \log |H| )$ generators, and so $|\operatorname{Aut}(H)| = O( |H| )^{O( \log |H| )}$.  Thus the kernel of $\phi$ - i.e. the centraliser $C_{G'}(H)$ of $H$ - has index $O( |H| )^{O( \log |H| )}$, and so the projection $\pi( C_{G'}(H) )$ of this centraliser has index $O( |H| )^{O( \log |H| )}$ in $G'$.  An easy application of the pigeonhole principle then yields a positive integer $N_i = O( |H| )^{O( \log |H| )}$ for each $1 \leq i \leq k$ such that $(e'_i)^{N_i}$ lies in this projection.  The group $\tilde G'$ generated by $(e'_i)^{N_i}$ has index $O( |H| )^{O( \log |H| )}$ in $G'$.  Thus, replacing $G'$ by this group if necessary (and replacing $H$ by $\tilde G' \cap H$, and $G$ by $\pi(\tilde G')$), we may assume that all the $e_i$ centralise $H$, i.e. $H$ is central, and we can now argue as in the case of trivial $\rho$ as before.   The author does not know, however, if the full strength of Lemma \ref{nil-lem} holds in the non-abelian case.
\end{remark}

Now we can prove Lemma \ref{virtnil}.  We begin with the case of coset nilprogressions.  We induct on the derived length $l$ of the coset nilprogression, as the case $l=0$ is vacuous.  We allow all implied constants to depend on $l,r_1,\ldots,r_l$.  We expand
$$ A = \phi_{l,0}(H_{l,0}) \cdot \ldots \cdot \phi_{l,r_l}(H_{l,r_l}) \cdot A_{l-1} \cdot \ldots \cdot A_1.$$
Observe from Definition \ref{cosdef} that $\phi_{l,0}(H_{l,0})$ is a finite abelian subgroup $H$ of $G$ that is normalised by $\langle A \rangle$, and that $\phi_{l,0}(H_{l,0}) \cdot \ldots \cdot \phi_{l,r_l}(H_{l,r_l})$ generates a larger abelian subgroup $V$ of $G$ that is also normalised by $\langle A \rangle$.  Observe that $H \subset A$, so in particular $|H| \leq |A|$.

Let $\pi: \langle A \rangle \to \langle A \rangle/V$ be the quotient map.  Observe that $\pi( A_{l-1} \cdot \ldots \cdot A_1 )$ is a coset nilprogression of derived length $l-1$, ranks $r_1,\ldots,r_{l-1}$, and cardinality $O(|A|)$, which generates the group $\langle A \rangle/V$.  Applying the induction hypothesis, we conclude that $\langle A \rangle/V$ contains a nilpotent subgroup $N$ of step $O(1)$, index $O(|A|^{O(1)})$, and generated by $O(1)$ generators.  

The group $\langle A \rangle/H$ projects onto $\langle A \rangle/V$ in the obvious manner; pulling back $N$, we obtain a subgroup $N'$ of $\langle A \rangle/H$ which still has index $O(|A|^{O(1)})$.  Since $V/H$ has $O(1)$ generators, we see that $N'$ is generated by $O(1)$ generators; as $N$ is nilpotent of step $O(1)$, we see from \eqref{dokey}, \eqref{soop} that $N'$ is also nilpotent of a slightly larger step $O(1)$.

Finally, as $\langle A \rangle$ projects onto $\langle A \rangle/H$, we can pull $N'$ back to a subgroup $N''$ of $\langle A \rangle$ of index $O(|A|^{O(1)})$.  Applying Lemma \ref{nil-lem} to the short exact sequence
$$ 0 \to H \to N'' \to N' \to 0$$
we obtain the claim.

In the case of nilprogressions, all finite abelian groups are trivial, and an inspection of the above argument shows that we can take $N = \langle A \rangle$ throughout the argument (as there is never any need to pass to a subgroup of finite index).  Indeed, the claim is significantly simpler in this case (Lemma \ref{nil-lem} does not need to be invoked).  We leave the details to the reader.

\section{The key proposition}\label{kps}

In this section we state the key proposition, describing sets of small ``doubling'' with respect to the action of an approximate group, and show how it implies Theorem \ref{frei-solv}.

\begin{proposition}[Key proposition]\label{key-prop}  Let $G = (G,\cdot)$ be a multiplicative group, and $V = (V,+)$ be an additive group.  Let $\rho: G \to \operatorname{Aut}(V)$ be an action of $G$ on $V$, let $K \geq 1$, let $A$ be a $K$-approximate group in $G$, and let $E$ be a centred subset of $V$ (see Appendix \ref{pse} for definitions).  Suppose also that 
\begin{equation}\label{ess}
 |\rho(A^{8})(2E)| \leq K |E|
\end{equation}
where 
$$ \rho(A)(E) := \{ \rho(a)(v): a \in A; v \in E \}.$$
Then there exists a coset progression
$$ H + P = H + \{ n_1 v_1 + \ldots + n_r v_r: a_i \in \Z, |n_i| \leq N_i \hbox{ for all } 1 \leq i \leq r \} $$
in $V$ of rank $r = O_K(1)$ and
$$ N_1 \geq N_2 \geq \ldots \geq N_r$$
that contains $E$ with
\begin{equation}\label{hpe}
 |H| |P| \ll_K |E|,
\end{equation}
as well as a centred subset $A'$ of $A^{16}$ with
$$ |A'| \gg_K |A|$$
such that for every $a \in A'$, we have $\rho(a)(H) = H$ and
$$ \rho(a) v_j = h_{a,j} + n_{a,j,1} v_1 + \ldots + n_{a,j,j-1} v_{j-1} + v_j$$
for all $1 \leq j \leq r$ and some $h_{a,j} \in H$ and integers $n_{a,j,i}$ for $1 \leq i \leq j-1$ with
$$ |n_{a,j,i}| \ll_K N_i/N_j.$$
Furthermore we have $H+P \subset O_K(1) \rho(A^8)(2E)$.
\end{proposition}

\begin{remark} In the converse direction, if $G, V, \rho, K, H, P$ are as in the above proposition, $A' \subset G$, and $E \subset H+P$ obeys \eqref{hpe}, then it is not hard to verify that $|\rho(A')(2E)| \ll_K |E|$.  Thus, up to constants, and the refinement of $A^{16}$ by a multiplicative factor, the above proposition is a tight description of sets $E$ which have small ``doubling'' with respect to an approximate group $A$.  The conclusion of this proposition can be viewed as a quantitative assertion that the action of $A'$ on $E$ is ``virtually unipotent''.  The exponents $8$ and $16$ can certainly be lowered, perhaps all the way to $2$, but we will not attempt to do so here.
\end{remark}

In the remainder of this section we show how this proposition implies Theorem \ref{frei-solv}.  To begin with we do not assume that $G$ is totally torsion-free, and discuss at the end of the section what changes when this additional assumption is added.

We induct on $l$.  The case\footnote{One could also use $l=0$ as the base case for the induction.} $l=1$ follows from Theorem \ref{gr2}, so suppose that $l \geq 2$ and that the theorem has already been proven for $l-1$.  

Fix $A,G,l,K$; we allow all implied constants to depend on $K,l$.  
By Lemma \ref{small}(i) we may assume without loss of generality that $A$ is a $O(1)$-approximate group.

Let $G^{(l)}$ be the $l^{th}$ commutator group in the derived series $G^{(1)} = G$, $G^{(i+1)} := [G^{(i)}, G^{(i)}]$, thus $G^{(l)}$ is an abelian normal subgroup of $G$.  Write $G' := G/G^{(l)}$, thus $G$ is an abelian extension of $G'$ (which is solvable of derived length at most $l-1$) by some abelian group $V = (V,+) \equiv G^{(l)}$, which we view additively.  Let $\pi: G \to G'$ be the projection map.  We shall abuse notation and identify $\ker(\pi)$ with $V$.

As in Section \ref{lamp-sec}, we define $E := \ker(\pi) \cap A^{20}$, thus by Lemma \ref{proj} $E$ is a centred subset of $V$ of cardinality $|E| \sim |A|/|\pi(A)|$.  The group $G$ acts on $V$ by conjugation:
$$ \rho(g)(v) := g v g^{-1},$$
and we have
$$ \rho(A^8)(2E) \subset \ker(\pi) \cap A^{100}$$
and thus by Lemma \ref{proj}
$$ |\rho(A^8)(2E)| \sim |E|.$$
Applying Proposition \ref{key-prop}, we can find a coset progression
$$ H + P = H + \{ n_1 v_1 + \ldots + n_r v_r: a_i \in \Z, |n_i| \leq N_i \hbox{ for all } 1 \leq i \leq r \} $$
in $V$ of rank $r = O(1)$ and
$$ N_1 \geq N_2 \geq \ldots \geq N_r$$
that contains $E$ with
\begin{equation}\label{hpe2}
|H| |P| \sim |E|
\end{equation}
and a centred subset $A'$ of $A^{16}$ of cardinality $|A'| \sim |A|$ such that for every $a \in A'$, we have 
\begin{equation}\label{aee}
\rho(a)(H) = H
\end{equation}
and
\begin{equation}\label{aee2}
\rho(a) v_j = h_{a,j} + n_{a,j,1} v_1 + \ldots + n_{a,j,j-1} v_{j-1} + v_j
\end{equation}
for all $1 \leq j \leq r$ and some $h_{a,j} \in H$ and integers $n_{a,j,i}$ for $1 \leq i \leq j-1$ with
$$ |n_{a,j,i}| \ll N_i/N_j.$$
Furthemore we have $H+P \subset O(1) \rho(A^2)(E)$.   In particular, $H+P \subset V \cap A^{O(1)}$.

By \eqref{aee2}, the additive group $\langle H+P\rangle \leq V$ generated by $H+P$ is invariant under the action of $\langle A'\rangle \leq G$ generated by $A'$; in particular, $\langle A'\rangle$ now acts on the quotient space $V/\langle H+P\rangle$.  This action is not necessarily trivial, but we can stabilise the most important portion of this action as follows.  More precisely, let $C \geq 1$ be a large integer (depending on $K,l$) to be chosen later, and consider the set $V \cap A^C$.  By Lemma \ref{proj}, this set has cardinality $O_C(|E|)$, and the sumset $(V \cap A^C) + H+P$ also has cardinality $O_C(|E|)$.  Since the coset progression $H+P$ itself has cardinality $O(|E|)$, we conclude that $V \cap A^C$ can be covered by $O_C(1)$ translates of $H+P$.  In particular, if we let $S \subset V / \langle H+P \rangle$ denote the image of $V \cap A^C$ under the quotient map $\pi_{H+P}: V \to V/\langle H+P\rangle$, we conclude that $|S| \ll_C 1$.

Let $\operatorname{Stab}(S)$ denote the subgroup of $\langle A'\rangle$ which fixes every single element of $S$.
We claim that this stabiliser group has large intersection with $(A')^2$, and more precisely that the set
$$ |\operatorname{Stab}(S) \cap (A')^2| \gg_C |A|.$$
To see this, consider the action of $A'$ on $S$.  Since $A' \subset A^{16}$, we see that the image of $S$ under this action must lie in the set $\pi_{H+P}(V \cap A^{C+O(1)})$, which has cardinality $O_C(1)$ by the same argument used to bound $|S|$.  Thus there are only $O_C(1)$ possible combinations for the way that a given element $a \in A'$ can act on $S$.  By the pigeonhole principle, we can thus find a subset $A''$ of $A'$ of size $|A''| \gg_C |A'| \sim |A|$ such that the action of any two elements of $A''$ on a single element of $S$ are identical.  This implies that any element of $A'' \cdot (A'')^{-1}$ must lie in $\operatorname{Stab}(S) \cap (A')^2$, and the claim follows.

Next, we claim that $(\operatorname{Stab}(S) \cap (A')^4)^2$ can be covered by $O(1)$ translates of $\operatorname{Stab}(S) \cap (A')^4$.  To see this, let $X \subset (\operatorname{Stab}(S) \cap (A')^4)^2$ be a maximal set such that the dilates $X \cdot (\operatorname{Stab}(S) \cap (A')^2)$ are disjoint.  This implies in particular that the dilates $X \cdot A'$ are disjoint.  But this dilates also lie in $(A')^9 \subset A^{144}$.  Since $|A^{144}|, |A'| \sim |A|$, we conclude that $|X| = O(1)$.  By maximality, we see that
$$ (\operatorname{Stab}(S) \cap (A')^4)^2 \subset X \cdot (\operatorname{Stab}(S) \cap (A')^2) \cap (\operatorname{Stab}(S) \cap (A')^2)^{-1} \subset X \cdot (\operatorname{Stab}(S) \cap (A')^4)$$
and the claim follows.  

Now let $B := \pi( \operatorname{Stab}(S) \cap (A')^4 )$.  From the above discussion we see that $(B,G')$ is a centred set of doubling constant $O(1)$.  (A key point here is that the doubling constant of $B$ is uniform in $C$.)  On the other hand, from Lemma \ref{proj} we have
$$ |B| \gg |\operatorname{Stab}(S) \cap (A')^4| / (|A|/|\pi(A)| ) \gg_C |\pi(A)|.$$
Since $G'$ is solvable of derived length at most $l-1$, we may apply the induction hypothesis and find a coset nilprogression $B' \subset B^{O(1)} \subset \pi(\operatorname{Stab}(S) \cap (A')^{O(1)})$ of derived length $l-1$, ranks $r_1,\ldots,r_{l-1}=O(1)$, and volume $\sim_C |\pi(A)|$ which $O(1)$-controls $B$, and in particular has cardinality $|B'| \gg_C |\pi(A)|$.

Observe that as $V$ is abelian, the action of $G$ on $V$ descends to an action of $G'$.  In particular, the elements of $B'$ act on $V$, and preserve $\langle H+P\rangle$, thus also act on $V/\langle H+P\rangle$.  By construction, all elements of $B'$ also preserve each element of $S$.

We have obtained a coset nilprogression $B'$ in $G'$ that relates to $\pi(A)$; we now need to build a coset nilprogression $A'$ in $G$ that is similarly related to $A$.  By Definition \ref{cosdef}, we have
$$ B' = \prod_{i=l-1}^{1} \prod_{j=0}^{r_i} \phi_{i,j}(H_{i,j})$$
for finite abelian groups $H_{i,0}$ and intervals $H_{i,j} = \{-N_{i,j},\ldots,N_{i,j}\}$ for $1 \leq i \leq l-1$ and $1 \leq j \leq r-1$ and maps $\phi_{i,j}: H_{i,j} \to G$, obeying the various properties in that definition.  

For each $1 \leq i \leq l-1$ and $0 \leq j \leq r_i$, the set $\phi_{i,j}(H_{i,j})$ is contained in $B'$, and hence in $\pi( \operatorname{Stab}(S) \cap (A')^{O(1)})$.  We may thus factor $\phi_{i,j} := \pi \circ \tilde \phi_{i,j}$ for some lifted map $\tilde \phi_{i,j}: H_{i,j} \to \operatorname{Stab}(S) \cap (A')^{O(1)}$; we can also ensure that $\tilde \phi_{i,j}(0) = 1$ for all $i,j$.  These maps $\tilde \phi_{i,j}$ will form the first $l-1$ levels of the coset nilprogression $A'$; we will fill in the final level of $A'$ shortly.

The various inclusions enjoyed by the $\phi_{i,j}$ in Definition \ref{cosdef} then lift to similar identities for $\tilde \phi$, modulo an error in $\ker(\pi) \equiv V$.  For instance, for any $1 \leq i \leq l-1$, $0 \leq j \leq r_i$, and $n,n' \in H_{i,j}$ with $n+n' \in H_{i,j}$ we have
$$ \tilde \phi_{i,j}(n) \tilde \phi_{i,j}(n') \in v_{i,j,n,n'} \cdot A_{l-1} \cdot \ldots \cdot A_{i+1} \cdot \tilde \phi_{i,j}(n+n')$$
for some $v_{i,j,n,n'} \in V$ where
$$ A_i := \tilde \phi_{i,0}(H_{i,0}) \cdot \ldots \cdot \tilde \phi_{i,r_i}(H_{i,r_i}).$$
Note that all factors other than $v_{i,j,n,n'}$ in the above inclusion are contained in $A^{O(1)}$, and thus $v_{i,j,n,n'} \in V \cap A^C$ if $C$ is large enough; thus
\begin{equation}\label{phi1}
 \tilde \phi_{i,j}(n) \tilde \phi_{i,j}(n') \in (V \cap A^C) \cdot A_{l-1} \cdot \ldots \cdot A_{i+1} \cdot \tilde \phi_{i,j}(n+n').
\end{equation}
In a similar spirit, we have
\begin{equation}\label{phi2}
[\tilde \phi_{i,j}(n), \tilde \phi_{i,j'}(n')] \in (V \cap A^C) \cdot A_{l-1} \cdot \ldots \cdot A_{i+1}
\end{equation}
for all $1 \leq i \leq l=1$, $0 \leq j,j' \leq r_i$, $n \in H_{i,j}$, and $n' \in H_{i,j'}$, and 
\begin{equation}\label{phi3}
[\tilde \phi_{i',j'}(n'), \tilde \phi_{i,j}(n)] \in (V \cap A^C) \cdot A_{\leq i-1} \cdot \tilde \phi_{i,0}(H_{i,0}) \cdot \ldots \cdot \tilde \phi_{i,\max(j-1,0)}(H_{i,\max(j-1,0)})
\end{equation}
for any $1 \leq i' < i \leq l$, $0 \leq j \leq r_i$, $0 \leq j' \leq r_{i'}$.

By definition of $S$, we have $\pi_{H+P}(V \cap A^C)$ for $1 \leq j \leq J$ lie in $S$.  In particular, for any $a \in \operatorname{Stab}(S) \cap \langle A' \rangle$ and $v \in V \cap A^C$, the action of $a$ on $v$ is trivial modulo $\langle H+P\rangle$, thus
$$(\rho(a)-1) (V \cap A^C) \subset \langle H+P \rangle.$$
On the other hand, we see from \eqref{aee2} that $\langle A' \rangle$ acts unipotently modulo $H$ on $\langle H+P\rangle$, in the sense that
$$ (\rho(a_1) - 1) \ldots (\rho(a_r)-1) \langle H+P \rangle \subset H$$
for all $a_1,\ldots,a_r \in \langle A' \rangle$.  Combining this with the previous observation, we thus have
$$ (\rho(a_1) - 1) \ldots (\rho(a_r)-1) (\rho(a_{r+1})-1) (V \cap A^C) \subset H$$
for all $a_1,\ldots,a_r \in A_{l-1} \cdot \ldots \cdot A_1$.  To exploit this, we define the sets $E_0,\ldots,E_{r+1} \subset V$ as
$$ E_j := H + \bigcup_{a_1,\ldots,a_j \in A_{l-1} \cdot \ldots \cdot A_1} (\rho(a_1)-1) \ldots (\rho(a_j)-1) (V \cap A^C),$$
thus $E_0 = V \cap A^C$, $E_j \subset V \cap A^{C+O(1)}$, $E_{r+1} = H$, and
\begin{equation}\label{eo}
 [a, v] \subset E_{\max(j+1,r+1)}
\end{equation}
for all $0 \leq j \leq r+1$, $a \in A_{l-1} \cdot \ldots \cdot A_1$, and $v \in E_j$.

We have
$$ E_j \subset V \cap A^{C+O(1)}; \quad H+P \subset V \cap A^{O(1)}$$
and thus
$$ |E_j + H+P| \ll |V \cap A^{C+O(1)}| \ll_C |A|/|\pi(A)|$$
thanks to Lemma \ref{proj}.  On the other hand, since $H+P$ contains $E$, we have $|H+P| \gg |A|/|\pi(A)|$.  We conclude that $E_j$ can be covered by $O_C(1)$ translates of $H+P$, indeed we may write
$$ E_j \subset \{ e_{j,1}, \ldots, e_{j,s_j} \} + H + P$$
for some $s_j = O_C(1)$ and $e_{j,1},\ldots,e_{j,s_j} \in E_j$, with $s_{r+1}=1$ and $e_{r+1,1}=0$.  From \eqref{eo} we see that
\begin{equation}\label{phi4}
[a,e_{j,i}] = e_{\max(j+1,r+1),i_{a,j,i}} + h_{a,j,i} +  n_{a,j,i,1} v_1 + \ldots + n_{a,j,i,r} v_{r}
\end{equation}
for all $a \in A_{l-1} \cdot \ldots \cdot A_1$, $0 \leq j \leq r+1$ and $1 \leq i \leq s_j$, and for some $1 \leq i_{a,j,i} \leq s_{\max(j+1,r+1)}$, some $h_{a,j,i} \in H$, and some integers $n_{a,j,i,k}$ with $|n_{a,j,i,k}| \ll N_k$ for $k=1,\ldots,r$.

We can now finish building the coset nilprogression $A'$.  Let $M > 1$ be a large integer (depending on $K,l,C$) to be chosen later.  We set
$$ r_l := r + s_r + \ldots + s_1$$
and define $H_{l,0},\ldots,H_{l,r_l}$ and the maps $\tilde \phi_{l,i}: H_{l,i} \to G$ as follows:
\begin{itemize}
\item $H_{l,0}$ is the finite abelian group $H$, and $\tilde \phi_{l,0}: H \to G$ is the inclusion map.
\item For $1 \leq i \leq r$, $H_{l,i}$ is the interval $\{ -M^{r-i+1} N_i, \ldots,M^{r-i+1} N_i\}$, and $\tilde \phi_{l,i}$ is the map $n \mapsto n v_i$.
\item For $0 \leq j \leq r$ and $1 \leq i \leq s_j$, if we set 
$$ i' := r + s_r + \ldots + s_{j+1} + i,$$
then $H_{l,i'}$ is the interval $\{ -1,0,1\}$, and $\tilde \phi_{l,i'}$ is the map $n \mapsto n e_{j,i}$.  
\end{itemize}
We set
$$ A_l := \tilde \phi_{l,0}(H_{l,0}) \cdot \ldots \cdot \tilde \phi_{l,r_l}$$
and
$$ A' := A_l \cdot \ldots \cdot A_1.$$
Using \eqref{phi1}, \eqref{phi2}, \eqref{phi3}, \eqref{phi4} and the abelian nature of $V$ (which contains all of $A_l$), we see (if $M$ is large enough) that $A'$ is a coset nilprogression of derived length $l$ and ranks $r_1,\ldots,r_l = O_C(1)$.  Also, from construction we have
$$ A' \subset A^{O_{M,C}(1)}.$$
Finally, from construction we also have
$$ |A'| \geq |B'| |H+P| \gg_C |\pi(A)| |E| \gg |A|,$$
Since $B'$ has volume $\sim_C |\pi(A)|$, $A'$ has volume
$$ \sim_{M,C} |\pi(A)| N_1 \ldots N_r |H| \sim |A|.$$
By Lemma \ref{small}(iii), we see that $(A')^{\pm 3}$ $O_{M,C}(1)$-controls $A$.  But from Corollary \ref{control-cor}, $A'$ $O_{C}(1)$-controls $(A')^{\pm 3}$.  We conclude that $A'$ $O_{M,C}(1)$-controls $A$, thus closing the induction.  This concludes the proof of Theorem \ref{frei-solv} assuming Proposition \ref{key-prop} when $G$ is not assumed to be totally torsion-free.

When instead $G$ is totally torsion-free, the abelian group $V$ is torsion-free and so the finite subgroup $H$ in that argument is trivial.  Also, $G/G^{(l)}$ is also totally torsion-free.  Thus in this case one can inductively eliminate the finite groups $H_{i,0}$ and maps $\phi_{i,0}$ from the coset nilprogression, replacing it by a nilprogression.

It remains to prove Proposition \ref{key-prop}.  This is the purpose of the next few sections of the paper.

\section{Obtaining a near-invariant set}

In the hypotheses of Proposition \ref{key-prop}, we have a set $E$ which expands by a multiplicative factor $K$ under the action of the set $A^4$.  The first step in the argument will be to modify $E$ and $A^4$ so that the expansion factor becomes extremely close to $1$, so that the (modification of the) set $E$ becomes almost invariant with respect to the action of (the modification of) $A^4$.  More precisely, we show

\begin{proposition}[Existence of near-invariant set]\label{near-invariant}  Let $G,V,K,A,E$ be as in Proposition \ref{key-prop}, and let $\eps > 0$.  Then there exists a set $E' \subset \rho(A^2) E$ with
$$ |E'| \gg_K |E|$$
and a centred subset $A' \subset A^4$ with
$$ |A'| \gg_{K,\eps} |A|$$
such that 
\begin{equation}\label{rhoke}
 |\rho(a)(E') \backslash E'| \leq \eps |E'|
\end{equation}
for all $a \in A'$.
\end{proposition}

\begin{remark}\label{rhok} From the triangle inequality we see from \eqref{rhoke} that
$$ |\rho(a)(E') \backslash E'| \leq j\eps |E'|$$
for all $j \geq 1$ and all $a \in (A')^j$.
\end{remark}

The remainder of this section is devoted to the proof of this proposition.  We fix $K$ and allow all implied constants to depend on $K$.
We may assume that $|E|$ is sufficiently large depending on $K, \eps$, since otherwise we could take $A' := A$ and $E' := \{0\}$.

Let $k_0$ be the first integer larger than $\eps^{-100}$.
Applying Proposition \ref{bz-prop2} (with $\eps$ replaced by $k_0^{-100}$), we can thus find a centred set $D \subset A^2$ with $|D| \gg_{\eps} |A|$ such that for every $1 \leq k \leq k_0$, one can find at least $(1-k_0^{-100}) |D|^k$ tuples $(d_1,\ldots,d_k) \in D^{\otimes k}$ such that
$$ d_1 \ldots d_k \in A^2.$$
In particular, we have
$$ \rho(d_1) \ldots \rho(d_k)(E) \subset \rho(A^2)(E)$$
for a fraction $1 - O(k_0^{-100})$ of the tuples $(d_1,\ldots,d_k) \in D^{\otimes k}$.

To motivate the argument which follows, let us temporarily ``cheat'' by pretending that $\eps$ vanishes, thus
$$ \rho(D^k) (E) \subset \rho(A^2)(E)$$
for all $0 \leq k \leq k_0$.  On the other hand, the cardinalities $|\rho(D^k)(E)|$ are increasing in $k$ and vary between $|E|$ and $|\rho(A^2)(E)| = O(|E|)$ (thanks to \eqref{ess}).  Thus, by the pigeonhole principle, we can find a centred set $E' := \rho(D^k) E$ for some $0 \leq k < k_0$ such that we have the near-invariance property $|\rho(D)(E')| \leq (1 + O(\frac{1}{k_0})) |E'|$, as required (cf. the selection of $V$ from the increasing sequence of subspaces $W_1,W_2,\ldots$ in Section \ref{lamp-sec}).

We now give a variant of the above argument in which we do not need to pretend that $\eps$ vanishes.

\begin{lemma}[Existence of a near-invariant set]  Let the notation and assumptions be as above.  Then there exists a centred set $E' \subset \rho(A^2) E$ with $|E'| \sim |E| \sim |A|/|\pi(A)|$ such that $\rho(d d')(v) \in E'$ for $1-O(k_0^{-1/10})$ of the triplets $(d,d',v) \in D \times D \times E'$.
\end{lemma}

\begin{remark} The quantity $k_0^{-1/10}$ can be improved here, but any expression which decays to zero as $k_0 \to \infty$ would suffice here.  The parameter $\eps$ has served its purpose with this lemma and will not appear in the rest of the argument.
\end{remark}

\begin{proof}  
For each $0 \leq k \leq k_0$, define the functions $f_k: V \to \R^+$ recursively by setting $f_0 = 1_E$ to be the indicator of $E$, and
\begin{equation}\label{fkv}
 f_{k+1}(v) := \frac{1}{|D|} \sum_{d \in D} f_k(\rho(d) v).
\end{equation}
From Young's inequality we see that the $l^2$ norms $\|f_k\|_{l^2(V)} := (\sum_{v \in V} |f_k(v)|^2)^{1/2}$ are non-increasing in $k$, and in particular that
\begin{equation}\label{fke}
 \|f_k\|_{l^2(V)} \leq |E|^{1/2}.
\end{equation}
On the other hand, we clearly have
$$ \|f_k\|_{l^1(V)} := \sum_{v \in V} |f_k(v)| = |E|.$$
We can also expand
$$ f_k = \frac{1}{|D|^k} \sum_{d_1,\ldots,d_k \in D} 1_{\rho(d_1) \ldots \rho(d_k) E}$$
(recall that $D$ is centred).  By construction, all but $O(k_0^{-100} |D|^k)$ of the summands are supported in $E_*$, thus 
\begin{equation}\label{l1v}
\|f_k\|_{l^1(V \backslash E_*)} \ll k_0^{-100} |E|
\end{equation}
and thus (if $\eps$ is small enough)
\begin{equation}\label{l2v}
 \|f_k\|_{l^1(E_*)} \gg |E|
\end{equation}
and hence by Cauchy-Schwarz (and \eqref{ess})
$$ \|f_k\|_{l^2(V)} \gg |E|^{1/2}.$$
Applying the pigeonhole principle, we can thus find $0 \leq k < k_0$ such that
$$ \|f_{k+1}\|_{l^2(V)} \geq (1 - O(k_0^{-1})) \|f_k\|_{l^2(V)}.$$
Squaring this using \eqref{fkv}, we see that
$$ \frac{1}{|D|^2} \sum_{d,d' \in D} \sum_{v \in V} f_k(\rho(d) v) f_k(\rho(d') v) \geq (1 - O(k_0^{-1})) \|f_k\|_{l^2(V)}^2.$$
On the other hand, from Cauchy-Schwarz we have
$$ \sum_{v \in V} f_k(\rho(d) v) f_k(\rho(d') v) \leq \|f_k\|_{l^2(V)}^2$$
for all $d,d' \in D$.  By Markov's inequality, we conclude that for $1 - O(k_0^{-1/2})$ of the pairs $(d,d')\in D \times D$
$$ \sum_{v \in V} f_k(\rho(d) v) f_k(\rho(d') v) \geq (1 - O(k_0^{-1/2})) \|f_k\|_{l^2(V)}^2,$$
and thus by the parallelogram law
$$ \sum_{v \in V} |f_k(\rho(d) v) - f_k(\rho(d') v)|^2 \ll k_0^{-1/2} \|f_k\|_{l^2(V)}^2$$
and thus (by \eqref{fke} and the centred nature of $D$) we see that 
\begin{equation}\label{fdd}
\sum_{v \in V} |f_k(\rho(dd') v) - f_k(v)|^2 \ll k_0^{-1/2} |E|
\end{equation}
for $1-O(k_0^{-1/2})$ of the pairs $(d,d') \in D \times D$.

On the other hand, from \eqref{fke}, \eqref{l1v} we see that
$$ |E| \ll \sum_{v \in V: f_k(v) \sim 1} |f_k(v)|^2 \leq \sum_{v \in V} |f_k(v)|^2 \leq |E|.$$
By the pigeonhole principle, we may thus find a threshold $\lambda \sim 1$ such that
$$ \sum_{v \in V: f_k(v) \geq \lambda} |f_k(v)|^2 \gg |E|$$
and
\begin{equation}\label{sss}
 \sum_{v \in V: \lambda - k_0^{-1/10} \leq f_k(v) \leq \lambda} |f_k(v)|^2 \ll k_0^{-1/10} |E|.
 \end{equation}
Fix this threshold, and let
$$ E' := \{ v \in \rho(A^2)(E): f_k(v) \geq \lambda \} \cup \{0\}.$$
Observe that $E'$ is centred and contained in $E_*$; from \eqref{l1v} (and the crude bound $\|f_k\|_{l^\infty(V)} \leq 1$) we see (if $\eps$ is small enough) that
$$ \sum_{v \in E'} |f_k(v)|^2 \gg |E|$$
and thus (again using the crude bound $\|f_k\|_{l^\infty(V)} \leq 1$)
$$ |E'| \gg |E|$$
and thus (from \eqref{ess}) $|E'| \sim |E| \sim |A|/|\pi(A)|$ as required.

Let $(d,d')$ be one of the pairs for which \eqref{fdd} holds.  Then by Markov's inequality, we have $|f_k(\rho(dd')v)-f_k(v)|^2 \ll k_0^{-1/4}$ for $O( k_0^{-1/4} |E| )$ values of $v$.  Thus for $(1 - O(k_0^{-1/4}))|E'|$ choices of $v \in E'$, we have $v \neq 0$ and 
$$ f_k(\rho(dd')v) = f_k(v) + O(k_0^{-1/8})$$
and thus $f_k(\rho(dd') v)$ is either greater than $\lambda$, or between $\lambda$ and $\lambda - k_0^{-1/10}$ (for $k_0$ large enough).  By \eqref{sss}, the latter only occurs for $O( k_0^{-1/10} |E| )$ values of $v$.  If the former occurs, then by \eqref{l1v} we see that $\rho(dd') v$ lies in $E'$ for all but at most $O( k_0^{-100} |E| )$ values of $v$.  Setting $\eps$ small enough, we obtain the claim.
\end{proof}

Let $E'$ be as in the above lemma.  Then, by construction, we have
$$ \sum_{d, d' \in D} |\rho(dd') E' \backslash E'| \ll k_0^{-1/10} |D|^2 |E|$$
and thus by Markov's inequality
$$ |\{ (d,d') \in D \times D: |\rho(dd') E' \backslash E'| \geq k_0^{-1/20} |E| \}| \gg |D|^2$$
which implies
$$ |\{ d \in D^2: |\rho(d) E' \backslash E'| \geq k_0^{-1/20} |E'| \}| \gg |D|.$$
If we denote the set on the left-hand side by $D'$, then $D'$ is symmetric, $D' \subset D^2 \subset A^4$ and $|D'| \gg_{\eps} |A|$, and 
$$ |\rho(d) E \backslash E| \ll k_0^{-1/20} |E|$$
for all $d \in D'$, and the claim follows from the choice of $k_0$.

\section{Obtaining a coset progression}

The next step is to replace the set $E'$ obtained in Proposition \ref{near-invariant} with a pair of proper coset progressions.

\begin{proposition}[Locating a good coset progression]\label{good-coset}  Let $G,V,K,A,E$ be as in Proposition \ref{key-prop}, and $F: \R^+ \to \R^+$ be a function.  Then there exists $1 \leq M \ll_{K,F} 1$ and a $F(M)$-proper coset progression $H+P \subset V$ containing $E$ of rank $O_K(1)$ and size
\begin{equation}\label{hle}
 |H| |P| \ll_{K,M} |E|,
\end{equation}
a subgroup $H'$ of $H$ of index $|H/H'| \ll_{K,M} 1$, an integer $1 \leq l \ll_{K,M} 1$, and a centred subset $A'$ of $A^4$ with
$$ |A'| \gg_{K,F} |A|$$
such that
\begin{equation}\label{joke}
 \rho((A')^{F(M)})( H'+P_l ) \subset H+P,
\end{equation}
where the refinement $P_l$ of $P$ is defined in Lemma \ref{fstp}.  Furthermore, we have $H+P \subset C\rho(A^2)(E)$ for some $C=O_K(1)$.
\end{proposition}

\begin{proof}  We allow all implied constants to depend on $K$.  Let $\eps > 0$ be a small number (depending on $K, F$) to be chosen later.  By Proposition \ref{near-invariant} (and Remark \ref{rhok}) we can find $E' \subset \rho(A^2)(E)$ with $|E'| \gg_K |E|$ and a centred set $A' \subset A^4$ such that
\begin{equation}\label{rhok2} |\rho(a)(E') \backslash E'| \leq j\eps |E'|
\end{equation}
for all $j \geq 1$ and all $a \in (A')^j$.

From \eqref{ess} we have $|2\rho(A^2)(E)| \ll |\rho(A^2)(E)|$.  Applying Theorem \ref{gr-proper}, we can place $\rho(A^2)(E)$ (and thus $E$ and $E'$) inside a $\tilde F(M)$-proper coset progression $H+P$ of rank $r = O_K(1)$ and size $|H| |P| = M |\rho(A^2)(E)|$ for some $1 \leq M \ll_{\tilde F} 1$, where $\tilde F$ is a function depending on $F$ and $K$ to be chosen later.  Furthermore we have $H+P \subset C \rho(A^2)(E)$ for some $C = O_{\tilde F}(1)$.

Applying Proposition \ref{sarkozy}, we can find a subgroup $H'$ of $H$ of cardinality $|H'| \gg_M |H|$ and $1 \leq m, l \ll_M 1$ such that every element of $H'+P_l$ has $\gg_M |E'|^{2m}$ representations of the form $v_1+\ldots+v_m-w_1-\ldots-w_m$ with $v_1,\ldots,w_m \in E'$.

The only property that has not yet been established is \eqref{joke}.  Let $a \in (A')^{\tilde F(M)}$.  Then every element of $\rho(a)(H'+P_l)$ has $\gg_M |E'|^{2m}$ representations of the form $v_1+\ldots+v_m-w_1-\ldots-w_m$ with $v_1,\ldots,w_m \in \rho(a)(E')$.  On the other hand, from \eqref{rhok2} we see that $\rho(a)(E')$ only differs from $E'$ by $O(\tilde F(M) \eps |E'|)$ elements.  Recall that $M = O_{\tilde F}(1)$.  Thus, if $\eps$ is sufficiently small depending on $\tilde F$ we see that every element of $\rho(a)(H'+P_l)$ can be represented in the form $v_1+\ldots+v_m-w_1-\ldots-w_m$ with $v_1,\ldots,w_m \in E'$.  Since $E' \subset H+P$, we conclude that $\rho(a)(H'+P_l) \subset H+2mP$.  The claim now follows by replacing $P$ by $2mP$ (and adjusting $M$ accordingly, and choosing $\tilde F$ sufficiently rapidly growing).
\end{proof}

The condition \eqref{joke} gives a tremendous amount of structural information on how $A'$ acts on $H+P$.  For instance, we have

\begin{corollary}[Existence of invariant torsion group]\label{itg}  Let the hypotheses, conclusion, and notation be as in Proposition \ref{good-coset}.  If $F$ is sufficiently rapidly growing (depending on $K$), then there exists a group $H' \leq H'' \leq H$ such that $\rho(a) H'' = H''$ for all $a \in A'$.  (In particular, $|H/H''| \ll_M 1$.)
\end{corollary}

\begin{proof}  From \eqref{joke} we see that for any $a \in (A')^{F(M)}$, $\rho(a)(H')$ is a finite subgroup in $H+P$.   Since $H+P$ is $F(M)$-proper, we conclude (if $F(M)$ is large enough) that $\rho(a)(H') \subset H$, thus $\rho( (A')^j )(H') \subset H$ for all $0 \leq j \leq F(M)$.

Now consider the groups $\langle \rho( (A')^j )(H') \rangle$.  The order of each such group is a multiple of the order of the previous group, and this order lies between $|H'|$ and $|H|$ for $0 \leq j \leq F(M)$.  Since $|H'| \gg_{K,M} |H|$, we thus conclude from the pigeonhole principle (if $F$ is sufficiently rapidly growing) that there exists $0 \leq j < F(M)$ such that
$$ \langle \rho( (A')^j )(H') \rangle =  \langle \rho( (A')^{j+1} )(H') \rangle.$$
Setting $H'' := \langle \rho( (A')^j )(H') \rangle$ we obtain the claim.
\end{proof} 

Corollary \ref{itg} partially describes the action of $A'$ on the $H$ component of the coset progression $H+P$.  Now we turn attention to the $P$ component.

\begin{corollary}[Virtually unipotent action modulo $H$]\label{itv}  Let the hypotheses, conclusion, and notation be as in Proposition \ref{good-coset}.  We express $P$ as
$$ P = \{ n_1 v_1 + \ldots + n_r v_r: n_i \in \Z, |n_i| \leq N_i \hbox{ for all } 1 \leq i \leq r \}.$$
Then if $F$ is sufficiently rapidly growing (depending on $K$), then for each $1 \leq i \leq r$ with $N_i \geq l^2$, and every $a \in A'$, we have
$$ \rho(a) (l v_i) = h_{a,i} + \sum_{j=1}^r n_{a,i,j} v_j$$
where $h_{a,i} \in H$ and each $n_{a,i,j}$ is an integer with
$$ |n_{a,i,j}| \ll_K N_j/N_i.$$
\end{corollary}

\begin{proof}  From \eqref{joke} we see that for any $a \in A'$, $\rho(a)(P_l)$ is a progression in $H+P$.  In particular, we have
$$ \rho(a)(l v_i) = h_{a,i} + \sum_{j=1}^r n_{a,i,j} v_j$$
for some $h_{a,i} \in H$ and some integers $n_{a,i,j}$ with $|n_{a,i,j}| \leq N_j$.  Furthermore, we have $n_i \rho(a)(l v_i) \in H+P$ for all $1 \leq n_i \leq N_i/l$.  Applying this (and using the fact that $H+P$ is $2$-proper) we obtain the claim.
\end{proof}

By repeatedly exploiting the pigeonhole principle, we can improve the conclusion of Corollary \ref{itv} by refining $A'$ by a constant factor:

\begin{proposition}[Genuinely unipotent action modulo $H''$]\label{itg2}  Let the hypotheses, conclusion, and notation be as in Proposition \ref{good-coset}, Corollary \ref{itg}, and Corollary \ref{itv}.  We arrange the dimensions $N_i$ of $P$ in decreasing order,
$$ N_1 \geq N_2 \geq \ldots \geq N_r.$$
Then there exists a centred set $A'' \subset (A')^2 \subset A^8$ with $|A''| \gg_K |A|$ such that for every $a \in A''$ and each $1 \leq i \leq r$ with $N_i \geq l^2$, we have
\begin{equation}\label{rhoo}
 \rho(a) (l v_i) = h_{a,i} + \sum_{j=1}^{i-1} n_{a,i,j} l v_j + l v_i
\end{equation}
where $h_{a,i} \in H''$ and each $n_{i,a,j}$ is an integer with
$$ |n_{a,i,j}| \ll_K N_j/N_i.$$
\end{proposition}

\begin{remark} One corollary of \eqref{rhoo} is that
$$ (\rho(A'')-1) \ldots (\rho(A'')-1) (H''+P_l) \subset H''$$
where $(\rho(A)-1)(E) := \{ \rho(a) v - v: a \in A, v \in E \}$.  Thus, in some sense, the action of $A''$ on $H''+P_l$ is unipotent modulo $H''$.
\end{remark}

\begin{proof}  We allow all implied constants to depend on $K$, thus for instance $l=O(1)$ and $r=O(1)$, and $|H/H''| = O(1)$.

We begin by analysing the action of $A'$ on $lv_1$.  From Corollary \ref{itv} and the decreasing nature of $N_i$ we see that the possible values of $\rho(a)(lv_1)$ as $a$ ranges over $A'$ can be covered by $O(1)$ translates of $H''$.  Thus, by the pigeonhole principle, we can find a subset $A_1$ of $A'$ with $|A_1| \gg |A'| \gg |A|$ such that $\rho(a)(l v_1) - \rho(a')(lv_1) \in H''$ for all $a,a' \in A_1$.  Since the action of $A'$ leaves the finite group $H''$ invariant, we may thus write
$$ \rho(a)(l v_1) - \rho(a')(lv_1) = \rho(a')(h_{a,a'}^{(1)})$$
for some $h_{a,a'}^{(1)} \in H''$.  We can rearrange this as
$$ \rho((a')^{-1} a)(lv_1) = h_{a,a'}^{(1)} + lv_1.$$
We thus conclude that for every $a$ in the centred set $A_1^{-1} \cdot A_1$, we have
\begin{equation}\label{rhoo2}
\rho(a)(lv_1) = h_a^{(1)} + lv_1
\end{equation}
for some $h_a \in H''$; note that this is the $i=1$ version of \eqref{rhoo}.

Now we analyse the action of $A_1$ on $lv_2$.  From Corollary \ref{itv} we see that the possible values of $\rho(a)(lv_2)$ as $a$ ranges over $A_1^{-1} \cdot A_1$ can be covered by $O(1)$ translates of $H'' + \{ l n_1 v_1: |n_1| \leq N_1/N_2 \}$.  Thus, by the pigeonhole principle, we may find a subset $A'_2$ of $A_1^{-1} \cdot A_1$ with $|A'_2| \gg |A|$ such that
$$ \rho(a)(l v_2) - \rho(a')(lv_2) \in H'' + \{ l n_1 v_1: |n_1| \leq 2N_1/N_2 \}$$
for all $a,a' \in A'_2$.  From \eqref{rhoo2} (and the fact that $\rho(a')$ preserves $H''$) we thus have
$$ \rho(a)(l v_2) - \rho(a')(lv_2) = \rho(a')( h_{a,a'}^{(2)} + n_{a,a',1}^{(2)} l v_1 )$$
for some $h_{a,a'}^{(2)} \in H''$ and an integer $n_{a,a',1}^{(2)}$ with $|n_{a,a',1}^{(2)}| \ll N_1/N_2$; we rearrange this as
$$ \rho((a')^{-1} a)(lv_2) = h_{a,a'}^{(2)} + n_{a,a',1}^{(2)} l v_1 + lv_2.$$
Thus for every $a \in (A'_2)^{-1} \cdot A'_2$, we have
\begin{equation}\label{rhov2}
\rho(a)(lv_2) = h_{a}^{(2)} + n_{a,1}^{(2)} l v_1 + lv_2
\end{equation}
for some $h_a^{(2)} \in H''$ and some integer $n_{a,1}^{(2)}$ of size $O(N_1/N_2)$.  This is not quite \eqref{rhoo} for $i=2$, because the set $(A'_2)^{-1} \cdot A'_2$ lives in $(A')^4$ rather than $(A')^2$.  But this can be easily fixed as follows\footnote{The same fix would allow one to substantially lower the exponents $8$ and $16$ in Proposition \ref{key-prop}; we omit the details.}.  Since $|A'_2|, |A_1^{-1} \cdot A_1|, |A_1| \sim |A|$ and $A'_2 \subset A_1^{-1} \cdot A_1$, we see from the pigeonhole principle that there exists $a_1 \in A_1$ such that 
$$ |A'_2 \cap ( a_1^{-1} \cdot A_1 )| \gg |A|.$$
Thus if we set $A_2 := (a_1 \cdot A'_2) \cap A_1$, then $A_2 \subset A_1 \subset (A')^2$ and $|A_2| \gg |A|$.  Also, we have $A_2^{-1} \cdot A_2 \subset (A'_2)^{-1} \cdot A'_2$, thus \eqref{rhov2} now holds for all $a \in A_2^{-1} \cdot A_2$.

We can continue in this fashion, repeatedly using the pigeonhole principle to build a nested sequence of sets $A_1 \supset A_2 \supset \ldots \supset A_r$ which obey more and more cases of \eqref{rhoo}.  Since $r=O(1)$, the final set $A_r$ will still have cardinality $\gg |A|$, and the claim follows (taking $A'' := A_r^{-1} \cdot A_r$).
\end{proof}

\section{Conclusion of the argument}

We are now ready to prove Proposition \ref{key-prop}.  Let $G,V,\rho,K,A,E$ be as in that proposition, let $F: \R^+ \to \R^+$ be a suitably rapidly growing function (depending on $K$), let $H, P, l, H''$ be as in Proposition \ref{good-coset} and Corollary \ref{itg}, and let $A'' \subset A^8$ be the set in Proposition \ref{itg2}.  We allow all implied constants to depend on $K$, thus $P$ has rank $r=O(1)$, $l=O(1)$, $|E| \sim |H+P| \sim |H''+P_l|$, and $|H/H''| = O(1)$.

By construction, we have
$$ E, H''+P_l \subset H+P.$$
We conclude that there exist $e_1,\ldots,e_J \in E$ with $J=O(1)$ such that
$$ E \subset \{e_1,\ldots,e_J\} + H'' + P_l.$$
By the pigeonhole principle, we can find $1 \leq j \leq J$ such that
$$ |E \cap (e_j + H'' + P_l)| \gg |E|.$$
From \eqref{ess} we have
$$ |\bigcup_{a \in A''} \rho(a)(E \cap (e_j + H'' + P_l))| \ll |E|.$$
From Corollary \ref{itg} and Proposition \ref{itg2}, we have 
$$\rho(A'')(H''+P_l) \subset H'' + C P_l$$
for some $C=O(1)$.  We conclude that each set $\rho(a)(E \cap (e_j + H'' + P_l))$ has cardinality $\sim |E|$ and is contained in $\rho(a)(e_j) + H'' + C P_l$.  By a greedy algorithm, we thus see that set $\{ \rho(a)(e_j): a \in A'' \}$ can be covered by $O(1)$ translates of $H'' + 2CP_l$, and hence by $O(1)$ translates of $H'' + P_l$.

Similarly, from \eqref{ess} we have
$$ |\bigcup_{i=1}^J \bigcup_{a \in A''} \rho(a) e_i + \rho(a)( E \cap (e_j + H'' + P_l) ) | \ll |E|$$
and so by arguing as before we see that $\{ \rho(a)(e_i) + \rho(a)(e_j): 1 \leq i \leq J; a \in A'' \}$ can also be covered by $O(1)$ translates of $H''+P_l$.  Subtracting, we conclude that 
$$\{ \rho(a)(e_i) : 1 \leq i \leq J; a \in A'' \} \subset \{ f_1,\ldots,f_M \} + H'' + P_l$$
for some $f_1,\ldots,f_M \in V$ with $M = O(1)$. In other words, we have
$$ \rho(a)(e_i) = f_{m_{a,i}} + x_{a,i}$$
for all $1 \leq i \leq J$ and $a \in A''$, where $1 \leq m_{a,i} \leq M$ and $x_{a,i} \in H''+P_l$.

Since $M, J = O(1)$, we may apply the pigeonhole principle and find integers $1 \leq m_1,\ldots,m_J \leq M$ and a subset $A'''$ of $A''$ with size $|A'''| \sim |A|$ such that
$$ \rho(a)(e_i) = f_{m_i} + x_{a,i}$$
for all $1 \leq i \leq J$ and $a \in A'''$.  In particular, for any $a,a' \in A'''$ and $1 \leq i \leq J$, we have
$$ \rho(a)(e_i) - \rho(a')(e_i) \in H'' + 2P_l.$$
Applying Corollary \ref{itg} and Proposition \ref{itg2}, we thus have
$$ \rho(a)(e_i) - \rho(a')(e_i) = \rho(a')(y_{a,a',i})$$
for some $y_{a,a',i} \in H'' + C P_l$ and some $C=O(1)$.  We rearrange this as
$$ \rho((a')^{-1} a )(e_i) = e_i + y_{a,a',i}.$$
Thus, if we set $A' := (A''')^{-1} \cdot A'''$, we have
$$ \rho(a)(e_i) = h_{a,i} + n_{a,i,1} v_1 + \ldots + n_{a,i,r} v_r + e_i$$
for all $a \in A'$ and $1 \leq i \leq J$, with $h_{a,i} \in H''$ and each $n_{a,i,j}$ being an integer of size $O(N_j)$.

Proposition \ref{key-prop} now follows by taking $H$ to be $H''$, and $P$ to be the progression 
$$ P_l + \{ s_1 e_1 + \ldots + s_J e_J: s_j \in \{-1,0,1\} \hbox{ for } 1 \leq j \leq J \}$$
(thus tacking $J$ dimensions of $1$ at the end of the existing dimensions $\{ N_j/l^2: N_j \geq l^2 \}$ of $P_l$; the various properties claimed in that proposition can be easily verified.  The proof of Proposition \ref{key-prop} (and thus Theorem \ref{frei-solv}) is complete.

\section{A Milnor-Wolf type theorem}\label{applied}

In this section we prove Theorem \ref{mwt}.   Let the notation and assumptions be as in that theorem.  We allow all implied constants to depend on $l,d$, and we assume that $R$ is sufficiently large depending on these parameters.  We may as well assume that $S$ is symmetric (since otherwise we just replace $S$ with $S \cup S^{-1}$).

By hypothesis, $|B_S(R)| \leq R^{O(1)}$.  By the pigeonhole principle, there thus exists a radius $R^{0.8} \leq r_0 \leq R^{0.9}$ (say) such that $B_S(r_0)$ has doubling constant $O(1)$.  Applying Theorem \ref{frei-solv}, we can find a coset nilprogression $A$ of cardinality and volume $\sim |B_S(r_0)|$, derived length $l$, and ranks $O(1)$ contained in $B_S(O(r_0))$ which $O(1)$-controls $B_S(r_0)$.  In particular, there exists a set $X \subset B_S(O(r_0))$ of cardinality $|X| = O(1)$ such that
$$ B_S(r_0) \subset X \cdot A.$$
Suppose that there exist two distinct elements $x, x'$ of $X$ such that $x \cdot A^{\pm 10}$ and $x' \cdot A^{\pm 10}$ overlap.  Then both $x \cdot A$ and $x' \cdot A$ are contained in $A^{\pm 21}$, and we have
$$ B_S(r_0) \subset (X \backslash \{x'\}) \cdot A^{\pm 21}.$$
Repeating this procedure $O(1)$ times, we eventually obtain a covering of the form
$$ B_S(r_0) \subset X' \cdot A^{\pm n}$$
for some $X' \subset X$ and some $n=O(1)$, such that the sets $x \cdot A^{\pm 10n}$ for $x \in X'$ are disjoint.  

Fix $X'$ and $n$.  For each $r$, let $X'_r$ be the set of all $x \in X'$ such that $x \cdot A^{\pm n}$ intersects $B_S(r)$.  These are subsets of $X'$ which increase in $r$, thus by the pigeonhole principle there exists $r_1 = O(1)$ such that $X'_{r_1} = X'_{r_1+1}$.

If $x \in X'_{r_1}$ and $s \in S$, then $x \cdot A^{\pm n}$ intersects $B_S(r)$, thus $sx \cdot A^{\pm n}$ intersects $B_S(r_1+1)$, thus $sx \cdot A^{\pm n}$ intersects $x' \cdot A^{\pm n}$ for some $x' \in X'_{r_1+1} \to X'_{r_1}$.  As the $x' \cdot A^{\pm 10n}$ are disjoint, we see that $x'$ is uniquely determined by $s$ and $x$; thus $s$ determines a map $\rho(s): X'_{r_1} \to X'_{r_1}$ from $X'_{r_1}$ to itself defined by $\rho(s)(x) := x'$.  In particular, $sx \cdot A^{\pm n}$ intersects $\rho(s)(x) \cdot A^{\pm n}$, and so
$$ sx \in \rho(s)(x) \cdot A^{\pm O(1)},$$
which implies that
\begin{equation}\label{xr1}
 S \cdot X'_{r_1} \subset X'_{r_1} \cdot A^{\pm O(1)}.
\end{equation}
We iterate this to obtain
$$ \langle S\rangle \cdot X'_{r_1} \subset X'_{r_1} \cdot \langle A \rangle.$$
Since $S$ generates $G$, we conclude that
$$ G = X'_{r_1} \cdot \langle A \rangle,$$
thus $\langle A \rangle$ has index $O(1)$ in $G$.  By Lemma \ref{virtnil}, we conclude that $G$ contains a nilpotent subgroup of $O(1)$ generators, step $O(1)$, and index $O( |A|^{O(1)} ) = O( R^{O(1)} )$.

Next, we return to \eqref{xr1} and iterate it again to obtain
$$ B_S(r) \cdot X'_{r_1} \subset X'_{r_1} \cdot A^{\pm O(r)}$$
for any $r \geq R$.  Since $X'_{r_1} \subset X \subset B_S(R)$, we conclude that
$$ B_S(r) \subset B_S(R) \cdot A^{\pm O(r)} \cdot B_S(R).$$
Applying Lemma \ref{nilpoly}, we have $|A^{\pm O(r)}| \ll r^{O(1)} |A| \ll r^{O(1)}$.  Since by hypothesis $|B_S(R)| \leq R^{O(1)} \ll r^{O(1)}$, we obtain $|B_S(r)| \ll r^{O(1)}$, and Theorem \ref{mwt} follows.

\begin{remark} It seems plausible that one could improve the polynomial growth bound $|B_S(r)| \ll r^{O(1)}$ for $r \geq R$ further, to establish a doubling bound $|B_S(2r)| \ll |B_S(r)|$ for $r \geq R$.  In order to do this, one presumably first needs to establish a doubling version of Lemma \ref{nilpoly}, but we will not pursue this matter.  (Because of virtual nilpotency of $G$, it is not difficult to see that in the asymptotic limit $r \to \infty$, we have $|B_S(r)| = (C + o(1)) r^D$ for some constants $C > 0$ and some integer $D = O(1)$, but this only settles the question in the case of extremely large $r$.)
\end{remark}

\appendix

\section{Product set estimates}\label{pse}

It is convenient to replace the notion of a set of small doubling by that of a $K$-approximate group (cf. \cite[Chapter 2]{tao-vu}, \cite{tao-noncomm})

\begin{definition}[Centred set] A \emph{centred set} in $G$ is a multiplicative set $(A,G)$ that contains the group identity $1$ (thus in particular $A^k \subset A^{k'}$ for $k \leq k'$) and is also symmetric (thus $A^{-1} := \{a^{-1}: a \in A \}$ is equal to $A$).
\end{definition}

\begin{definition}[$K$-approximate group]\label{approx}  Let $K \geq 1$.  A \emph{$K$-approximate group} is a centred set $(A,G)$ such that there exists a set $X \subset G$ with $|X| \leq K$ such that $A^2 \subset X \cdot A \subset A \cdot X \cdot X$ and $A^2 \subset A \cdot X \subset X \cdot X \cdot A$.  Note that this in particular implies that
$$ |A^k| \leq K^{k-1} |A|$$
for all $k \geq 1$.
\end{definition}

\begin{examples}  Finite subgroups of $G$ are $1$-approximate groups.  Geometric progressions are $2$-approximate groups.  Balls in nilpotent groups of bounded rank (with respect to the word metric on a bounded set of generators) are $O(1)$-approximate groups.  The homomorphic image of a $K$-approximate group is another $K$-approximate group.  The Cartesian product of two $K$-approximate groups is a $KK'$-approximate group.
If $\pi: H \to G$ is a finite extension of $G$, and $A$ is a $K$-approximate group in $G$, then $\pi^{-1}(A)$ is a $K$-approximate group in $H$.  If $A$ is a $K$-approximate group, then $A^k$ is a $K^k$-approximate group for any $k \geq 1$, and $A^k$ and $A$ $K^{k-1}$-control each other.
\end{examples}

\begin{lemma}[Small doubling controlled by $K$-approximate groups]\label{small}  Let $K \geq 1$.
\begin{itemize}
\item[(i)] If $(A,G)$ is a multiplicative set of doubling constant at most $K$, then there exists a $O_K(1)$-approximate group $B \subset A^{\pm 3}$ that $O_K(1)$-controls $A$.
\item[(ii)] If $(A,G)$ is a multiplicative set of tripling constant at most $K$, then $A^{\pm 3}$ is a $O_K(1)$-approximate group that $O_K(1)$-controls $A$.
\item[(iii)] If $(A,G)$ is a $K$-approximate group, and $A' \subset A$ is such that $|A'| \geq |A|/K$, then $(A')^{\pm 3}$ is a $O_K(1)$-approximate group which $O_K(1)$-controls $A$.
\end{itemize}
\end{lemma}

\begin{proof}  For (i), see \cite[Theorem 4.6]{tao-noncomm}.  For (ii), see \cite[Corollary 3.10]{tao-noncomm}.  For (iii), see \cite[Corollary 3.10]{tao-noncomm} and \cite[Lemma 3.6]{tao-noncomm}.  Indeed, the bounds here are polynomial in $K$, although we will not exploit this.  
\end{proof}

Now, we investigate how approximate groups behave in group extensions.

\begin{lemma}[Projection lemma]\label{proj}  Let $\tilde G$ be an extension of a group $G$ with projection map $\pi: \tilde G \to G$, and let $(\tilde A, \tilde G)$ be a $K$-approximate group for some $K \geq 1$.
\begin{itemize}
\item[(i)] $\pi(\tilde A)$ is a $K$-approximate group.
\item[(ii)] For any $h \in \pi(\tilde A)$ and $k \geq 3$, we have $|\pi^{-1}(\{h\}) \cap (\tilde A)^k| \sim_{K,k} |\tilde A|/|\pi(\tilde A)|$.  
\item[(iii)] For any $k \geq 3$, $(\ker(\pi) \cap (\tilde A)^k)^3$ is a $O_{K,k}(1)$-approximate group of cardinality $\sim_{K,k} |\tilde A|/|\pi(\tilde A)|$.
\end{itemize}
\end{lemma}

\begin{proof}  Part (i) follows from the more general observation that the homomorphic image of a $K$-approximate group is another $K$-approximate group.  To prove (ii), if we set $E := \pi^{-1}(\{h\}) \cap (\tilde A)^k$, then we see that 
$$|\tilde A \cdot E| \leq |(\tilde A)^{k+1}| \ll_{K,k} |\tilde A|.$$
Splitting $\tilde A$ into the fibres $\pi^{-1}(\{a\})$ of $\pi$, we conclude that
$$ |\pi(\tilde A)| |E| \ll_{K,k} |\tilde A|$$
thus yielding the upper bound for (ii).  To obtain the lower bound, observe from the pigeonhole principle that there must exist $h_0 \in \pi(\tilde A)$ such that
$$ |\pi^{-1}(\{h_0\}) \cap \tilde A| \geq |\tilde A| / |\pi(\tilde A)|.$$
Taking quotient sets, we conclude that
$$ |\ker(\pi) \cap \tilde A^2| \geq |\tilde A| / |\pi(\tilde A)|$$
and then multiplying by an arbitrary element of $\pi^{-1}(\{h\}) \cap \tilde A$ we obtain the claim.

Finally, from (ii) we see that $\ker(\pi) \cap (\tilde A)^k$ is a centred multiplicative set of tripling constant $O_{K,k}(1)$ and cardinality $\sim_{K,k} |\tilde A|/|\pi(\tilde A)|$, and (iii) then follows from Lemma \ref{small}(ii).
\end{proof}

\begin{remark} For a more refined statement about the structure of $\tilde A$ in terms of $\pi(\tilde A)$ and $\ker(\pi) \cap (\tilde A)^k$, see \cite[Lemma 7.7]{tao-noncomm}.
\end{remark}

The following result follows immediately from the main result in \cite{gr-4}; the case $G=\Z$ is of course Freiman's original theorem\cite{frei}; the case when $G$ has bounded torsion is in \cite{ruzsa-group}.

\begin{definition}[Generalised arithmetic progression]  A \emph{(symmetric) generalised arithmetic progression} (or \emph{progression} for short) in an abelian group $G = (G,+)$ is any set $P$ of the form
$$ P = \{ n_1 v_1 + \ldots + n_r v_r: n_i \in \Z, |n_i| \leq N_i \hbox{ for all } 1 \leq i \leq r \}$$
for some $v_1,\ldots,v_r \in \Z$ and $N_1,\ldots,N_r \geq 1$.  We refer to $r$ as the \emph{rank} of the progression.  If $t \geq 1$, we say that the progression $P$ is \emph{$t$-proper} if the sums
$$ n_1 v_1 + \ldots + n_r v_r$$
for $n_i \in \Z$, $|n_i| \leq tN_i$ are all distinct.  We say that a progression is \emph{proper} if it is $1$-proper.
\end{definition}

\begin{remark} Technically, the progression should refer to the data $(r, v_1,\ldots,v_r, N_1,\ldots,N_r)$ rather than the set $P$ (since there are multiple ways to represent a single set as a progression, but we shall abuse notation and use the set $P$ to denote the progression instead.  Similarly for the concept of a coset progression below.
\end{remark}

\begin{definition}[Coset progression]  A \emph{(symmetric) coset progression} in an abelian group $G = (G,+)$ is any set of the form $H+P$, where $H$ is a finite subgroup $H$ of $G$, and $P$ is a generalised arithmetic progression
$$ P = \{ n_1 v_1 + \ldots + n_r v_r: n_i \in \Z, |n_i| \leq N_i \hbox{ for all } 1 \leq i \leq r \}.$$
We refer to $r$ as the \emph{rank} of the coset progression.  If $t \geq 1$, we say that the coset progression $H+P$ is \emph{$t$-proper} if the sums
$$ h + n_1 v_1 + \ldots + n_r v_r$$
for $h \in H$, $n_i \in \Z$, $|n_i| \leq tN_i$ are all distinct.  We say that a coset progression is \emph{proper} if it is $1$-proper.
\end{definition}

Coset progressions have already appeared in Theorems \ref{gr}, \ref{gr2} in the introduction.  We now give a variant of Theorem \ref{gr} that ensures some properness to the coset progressions.

\begin{theorem}[Proper Green-Ruzsa's Freiman theorem]\label{gr-proper}\cite{gr-4}  Let $(A,G)$ be an additive set with doubling constant at most $K$ for some $K \geq 1$. For any function $F: \R^+ \to \R^+$, we can find $1 \leq M \ll_{K,F} 1$ and a $F(M)$-proper coset progression $H+P$ of rank $O_K(1)$ and cardinality $M|A|$ that contains $A$.  Furthermore we have $H+P \subset C(A \cup \{0\} \cup -A)$ for some $C = O_{K,F}(1)$.
\end{theorem}

\begin{proof}   This will follow from Theorem \ref{gr} by a standard ``rank reduction argument'' which we now give.  We run the following algorithm to find $M$ and $H+P$.

\begin{itemize}
\item Step 0.  By Theorem \ref{gr}, we can find a coset progresion $H+P$ of rank $O_K(1)$ and cardinality $O_K(A)$ that contains $A$, with $H+P \subset C(A \cup \{0\} \cup -A)$ for some $C = O_K(1)$.
\item Step 1.  Set $M := |H+P|/|A|$, and set $r$ to be the rank of $H+P$.  If $H+P$ is already $F(M)$-proper, then {\tt STOP}.  Otherwise, move on to Step 2.
\item Step 2.  Since $H+P$ is not $F(M)$-proper, we may invoke \cite[Lemma 5.1]{john}, and contain $H+P$ in a coset progression $H'+P'$ of rank at most $r-1$ and cardinality $O_{F,M,r}(|H+P|)$.  Furthermore we have $H'+P' \subset C(H+P)$ for some $C = O_{r,F(M)}(1)$.
\item Step 3.  Replace $H+P$ by $H'+P'$ and return to Step 1.
\end{itemize}

Observe that the rank of $H+P$ decreases by at least $1$ each time one passes from Step 3 to Step 1, and so the algorithm terminates in $O_K(1)$ steps.  The quantity $M$ increases to $O_{F,M,r}(1)$ in each iteration, so is ultimately bounded by $O_{F,K}(1)$, and the claim follows.
\end{proof}

\section{S\'ark\"ozy type theorems}

A classical result of S\'ark\"ozy\cite{sar} asserts that if $A \subset \{1,\ldots,N\}$ and $k|A| \geq CN$ for some sufficiently large absolute constant $C$, then $kA$ contains an arithmetic progression of length comparable to $N$; see \cite{lev} for a more precise result, and \cite{szvu} for
various generalisations and extensions.  In this appendix we give several results of this type, in which $A$ now lives in a finite group, generalised arithmetic progression, or coset progression.

\begin{lemma}[S\'ark\"ozy-type theorem in finite abelian groups]\label{fsar}  Let $G = (G,+)$ be a finite abelian group, and let $A \subset G$ be such that $|A| \geq \delta |G|$ for some $0 < \delta < 1$.  Then there exists a positive integer $1 \leq m \ll_\delta 1$ and a subgroup $H$ of $G$ of index $|G/H| \ll_\delta 1$ such that $mA-mA$ contains $H$.  In fact, we have the stronger statement that each element $h$ of $H$ has $\gg_\delta |G|^{2m}$ representations of the form $h = a_1 + \ldots + a_m - b_1 - \ldots - b_m$ with $a_1,\ldots,a_m,b_1,\ldots,b_m \in A$.
\end{lemma}

\begin{remark} This should be compared with Bogulybov's theorem (or Theorem \ref{gr}(i)), in which one can take $m=2$, but the subgroup $H$ is replaced by a coset progression.
\end{remark}

\begin{proof} We use Fourier analysis.  By increasing $\delta$ if necessary, we may take $|A|=\delta|G|$.  Let $\hat G$ be the Pontryagin dual, consisting of all homomorphisms $x \mapsto \xi \cdot x$ from $G$ to $\R/\Z$.  We let $\eps > 0$ be a small quantity depending on $\delta$ to be chosen later, and introduce the spectrum
$$ \Sigma := \{ \xi \in \hat G: |\hat 1_A(\xi)|^2 > (1-\eps) \delta^2 \}$$
where 
$$ \hat 1_A(\xi) := \frac{1}{|G|} \sum_{x \in G} 1_A(x) e(-\xi \cdot x)$$
is the Fourier transform of $A$, and $e(x) := e^{2\pi i x}$.  We can write
$$ |\hat 1_A(\xi)|^2 := \frac{1}{|G|^2} \sum_{x,y \in A} \cos(2\pi (\xi \cdot (x-y)))$$
and
$$ \delta^2 - |\hat 1_A(\xi)|^2 := \frac{1}{|G|^2} \sum_{x,y \in A} (1 - \cos(2\pi (\xi \cdot (x-y)))).$$

Let $\xi \in \Sigma$.  Then from the elementary estimate
$$ 1-\cos(2\pi \theta) \sim \dist(\theta,\Z)^2$$
for any $\theta$, we have
$$ 1-\cos(2\pi n \theta) \lesssim n^2 (1 - \cos(2\pi \theta)$$
and hence
$$ \delta^2 - |\hat 1_A(n \xi)|^2 \ll n^2 \eps \delta^2$$
for any integer $n \geq 1$, and in particular that
$$ |\hat 1_A(n \xi)| \sim \delta$$
for all $1 \leq n \leq c / \sqrt{\eps}$ for some absolute constant $c > 0$.  Combining this with Plancherel's theorem, we see (if $\eps$ is small enough depending on $\delta$) that these $n\xi$ cannot all be distinct, and so we conclude that every $\xi \in \Sigma$ has order $O_\delta(1)$.  On the other hand, another application of Plancherel shows that $|\Sigma| \ll_\delta 1$.  We conclude that the subgroup $\langle \Sigma \rangle$ of the abelian group $\hat G$ generated by $\Sigma$ also has cardinality $O_\delta(1)$.  Thus, if we let 
$$H := \{ x \in G: \xi \cdot x = 0 \hbox{ for all } \xi \in \langle \Sigma \rangle \}$$
be the orthogonal complement of $\langle \Sigma \rangle$, then (by another application of Plancherel) we see that $|H| \gg_\delta |G|$.

Now let $m$ be a large integer depending on $\delta,\eps$ to be chosen later, and let $h \in H$.  The number of representations of $h$ as $a_1+\ldots+a_m-b_1-\ldots-b_m$, as is well known, can be expressed via the Fourier transform as
$$ |G|^{2m} \sum_{\xi \in \hat G} |\hat 1_A(\xi)|^{2m} e(h \cdot \xi).$$
The contribution of the $\xi=0$ term is $\delta^{2m} |G|^{2m}$.  The contribution of the other elements of $\Sigma$ is non-negative.  As for the remaining contributions, one can bound them by
$$ O( |G|^{2m} (1-\eps)^{m-1} \delta^{2m-2} \sum_{\xi \in \hat G} |\hat 1_A(\xi)|^2 ) = O( |G|^{2m} (1-\eps)^{m-1} \delta^{2m-1} )$$
thanks to Plancherel.  If $m$ is large enough, this error term is dominated by the main term, and the claim follows.
\end{proof}

\begin{lemma}[S\'ark\"ozy type theorem in progressions]\label{fstp}  Let $G = (G,+)$ be an abelian group, and let
$$ P = \{ n_1 v_1 + \ldots + n_r v_r: n_i \in \Z, |n_i| \leq N_i \hbox{ for all } 1 \leq i \leq r \}$$
be a proper generalised arithmetic progression in $G$ of rank $r$.  Let $A \subset P$ be such that $|A| \geq \delta |P|$ for some $0 < \delta < 1$.  Then there exists positive integers $1 \leq m, l \ll_{\delta,r} 1$ such that $P_l \subset mA - mA$, where $P_l$ is the generalised arithmetic progression
$$ P_l = \{ l n_1 v_1 + \ldots + l n_r v_r: n_i \in \Z, |n_i| \leq N_i / l^2 \hbox{ for all } 1 \leq i \leq r \}.$$
In fact, we have the stronger statement that each element $v$ of $P_l$ has $\gg_{\delta,r} |P|^{2m}$ representations of the form $v = a_1 + \ldots + a_m - b_1 - \ldots - b_m$ with $a_1,\ldots,a_m,b_1,\ldots,b_m \in A$.
\end{lemma}

\begin{remark} This result can probably deduced from the powerful results in \cite{szvu} (after using Freiman isomorphisms to map the progression into the integers), however for sake of self-containedness we present a Fourier-analytic proof (analogous to the proof of Lemma \ref{fsar}) here.
\end{remark}

\begin{proof}  We will fix $\delta,r$ and allow all constants to depend on these parameters.  

It suffices to establish the case when $G = \Z^r$ and $v_1,\ldots,v_r$ is the standard basis of $\Z^r$, since the general case can then be obtained by applying the homomorphism $(n_1,\ldots,n_r) \mapsto n_1 v_1 + \ldots + n_r v_r$ from $\Z^r$ to $G$.  (Note that the presence of a kernel in this homomorphism will work in one's favour, since $P$ is proper.)  

The Pontryagin dual of $\Z^r$ is the torus $(\R/\Z)^r$, and we have the Fourier transform
$$ \hat 1_A(\xi) := \sum_{x \in A} e( -\xi \cdot x )$$
(note the conventions here are slightly different from that in the case of finite groups $G$).  

Let $\eps > 0$ be a small number (depending on $\delta,r$) to be chosen later, let $l \geq 1$ be a large integer (depending on $\eps, \delta, r$) to be chosen later; let $\eps' > 0$ be an even smaller number (depending on $l,\eps,\delta,r$) to be chosen later, and let $m \geq 1$ be a large integer (depending on $\eps',l,\eps,\delta,r$) to be chosen later.  We introduce the spectrum $\Sigma \subset (\R/\Z)^r$, defined by
$$ \Sigma := \{ \xi \in (\R/\Z)^r: |\hat 1_A(\xi)|^2 > (1-\eps) |A|^2 \}.$$
This is an open subset of $(\R/\Z)^r$.  Observe that it contains the box $R_{c\eps}$ for some small constant $c = c_r > 0$, where
$$ R_{s} := \{ (\xi_1,\ldots,\xi_r): \|\xi_i\| \leq s / N_i \hbox{ for all } 1 \leq i \leq r \}$$
and $\|\xi\|$ is the distance from $\xi \in \R/\Z$ to the origin.

Now let $\xi \in \Sigma$.  Arguing as in Lemma \ref{fsar} we see that
$$ |\hat 1_A(n\xi)| \sim |A|$$
for all integers $n$ with $|n| \leq c\eps^{-1/2}$ for some absolute constant $c$. Thus we can find bounded coefficients $|\alpha_n| \leq 1$ for $|n| \leq c \eps^{-1/2}$ such that
$$ |\sum_{|n| \leq c \eps^{-1/2}} \alpha_n \hat 1_A(n \xi)| \gg \eps^{-1/2} |A|.$$
The left-hand side can be rearranged as
$$ |\sum_{x \in A} \sum_{|n| \leq c \eps^{-1/2}} \alpha_n e(-n \xi \cdot x)| \gg \eps^{-1/2} |A|.$$
By Cauchy-Schwarz we conclude that
$$ \sum_{x \in P} |\sum_{|n| \leq c \eps^{-1/2}} \alpha_n e(-n \xi \cdot x)|^2 \gg \eps^{-1} |A| \sim \eps^{-1} |P|.$$
Observe that if $0 < s \leq 1$ is a sufficiently small absolute constant, then
$$ |\int_{R_s} e(\xi \cdot x)\ d\xi| \gg_s 1/|P|$$
for all $x \in P$.  We conclude that
$$ \sum_{x \in \Z^r} |\int_{R_s} e(\xi \cdot x)\ d\xi|^2 
|\sum_{|n| \leq c \eps^{-1/2}} \alpha_n e(-n \xi \cdot x)|^2 \gg_s \eps^{-1} / |P|.$$
By Fourier analysis, the left-hand side can be rearranged as
$$ \sum_{n, n': |n|, |n'| \leq c\eps^{-1/2}} \alpha_n \overline{\alpha_{n'}} \operatorname{mes}( (n\xi + R_s) \cap (n'\xi + R_s) ).$$
The summand is bounded by $O(1/|P|)$, and vanishes unless $(n-n')\xi \in R_s$.  We conclude that
$$ | \{ (n,n'): |n|, |n'| \leq c\eps^{-1/2}; (n-n') \xi \in R_s \} | \gg_s \eps^{-1}.$$
If $\eps$ is sufficiently small, we conclude that there exists $1 \leq n'' \ll \eps^{-1/2}$ such that $n'' \xi \in R_s$.  If we let $M = O_\eps(1)$ be the least common multiple of all integers $1 \ll n'' \ll \eps^{-1/2}$, and let $\Gamma \subset (\R/\Z)^r$ be the finite subgroup of the torus $(\R/\Z)^r$ consisting of the $M^{th}$ roots of unity, we conclude that
\begin{equation}\label{sigmal}
\Sigma \subset \Gamma + R_s.
\end{equation}
We can require that $l$ is a multiple of $M$.

Now let $v \in P_l$.  The number of representations of $v$ as $a_1+\ldots+a_m-b_1-\ldots-b_m$ can be expressed via Fourier analysis as
$$ \int_{(\R/\Z)^r} |\hat 1_A(\xi)|^{2m} e( v \cdot \xi )\ d\xi;$$
since this number is clearly real, we can also express it as
$$ \int_{(\R/\Z)^r} |\hat 1_A(\xi)|^{2m} \cos( 2\pi v \cdot \xi )\ d\xi.$$
We consider the contributions of various portions of this integral.  For $\xi$ in the rectangle $R_{\eps'}$, we have $|\hat 1_A(\xi)| \geq (1 - O(\eps')) |A|$ and $\cos(2\pi h \cdot \xi) \gg 1$, and so the contribution of this rectangle is
$$ \gg (1 - O(\eps'))^{2m} |A|^{2m} |R_{\eps'}| \gg_{\eps'} (1 - O(\eps'))^{m} |A|^{2m-1}.$$
Next, for $\xi \in \Sigma \backslash R_{\eps'}$, we see from \eqref{sigmal} and the definition of $P_l$ that $\cos(2\pi v \cdot \xi) \gg 1$, and so the contribution of this region is non-negative.  Finally, we turn to the contribution when $\xi \not \in \Sigma$, which we can bound by
$$ O( (1-\eps)^{m-1} |A|^{2m-2} \int_{(\R/\Z)^r} |\hat 1_A(\xi)|^2\ d\xi ) = O( (1-\eps)^{m-1} |A|^{2m-1} )$$
thanks to Plancherel's theorem.  If we choose $m$ large enough depending on $\eps, \eps'$, the error term is dominated by the main term, and the claim follows.
\end{proof}

We now unify the above two lemmas into a single proposition.

\begin{proposition}[S\'ark\"ozy-type theorem in coset progressions]\label{sarkozy} Let $G = (G,+)$ be an abelian group, and let
$$ H+P = H+\{ n_1 v_1 + \ldots + n_r v_r: n_i \in \Z, |n_i| \leq N_i \hbox{ for all } 1 \leq i \leq r \}$$
be a proper generalised arithmetic progression in $G$ of rank $r$.  Let $A \subset H+P$ be such that $|A| \geq \delta |H| |P|$ for some $0 < \delta < 1$.  Then there exists a subgroup $H'$ of $H$ with cardinality $|H'| \gg_{\delta,r} |H|$ and positive integers $1 \leq m, l \ll_{\delta,r} 1$ such that $H'+P_l \subset mA - mA$, where $P_l$ is as in Lemma \ref{fstp}.  
In fact, we have the stronger statement that each element $h+v$ of the coset progression $H'+P_l$ has $\gg_{\delta,r} |H|^{2m} |P|^{2m}$ representations of the form $h+v = a_1 + \ldots + a_m - b_1 - \ldots - b_m$ with $a_1,\ldots,a_m,b_1,\ldots,b_m \in A$.
\end{proposition}

\begin{proof}  The arguments are basically a pastiche of those used to prove Lemma \ref{fsar} and Lemma \ref{fstp}.  As there are no new ideas here, we give only a sketch of the proof.  We allow all implied constants to depend on $\delta,r$.

As in Lemma \ref{fstp}, we may assume that $G = H \times \Z^r$, and that $v_1,\ldots,v_r$ are the basis vectors of $\Z^r$.  The Pontryagin dual is then $\hat G = \hat H \times (\R/\Z)^r$, with Fourier transform
$$ \hat 1_A( \xi, \eta ) := \frac{1}{|H|} \sum_{(h,x) \in A} e( - \xi \cdot h ) e( - \eta \cdot x )$$
for $\xi \in \hat H$ and $\eta \in (\R/\Z)^r$.  We give $\hat G$ the obvious Haar measure $dm$, being the product of counting measure on $\hat H$ and normalised Haar measure on $(\R/\Z)^r$.

Now let $\eps > 0$ be a small quantity depending on $\delta$ to be chosen later, and let $\Sigma \subset \hat G$ be the set
$$ \Sigma := \{ (\xi,\eta) \in \hat G: |\hat 1_A(\xi,\eta)|^2 > (1-\eps) |A|^2 / |H|^2 \}.$$
Repeating the arguments in Lemma \ref{fstp} with minor changes, we see that if $\eps$ is sufficiently small depending on $\delta$, then
$$ \Sigma \subset \Gamma + (\{0\} \times R_s)$$
for some small absolute constant $s > 0$, where $\Gamma := \{ (\xi,\eta) \in \hat G: M(\xi,\eta) = 0 \}$ and $1 \leq M \ll_\eps 1$ is an integer.

The arguments in Lemma \ref{fstp} also show that if $(\xi_1,\eta_1), \ldots, (\xi_n,\eta_n) \in \Sigma$ are such that $(\xi_i,\eta_i) - (\xi_j,\eta_j) \not \in \{0\} \times R_s$ for all $1 \leq i < j \leq n$, then $n \ll_\eps 1$.  From this we conclude in particular that the projection $\pi(\Sigma)$ from $\hat H \times (\R/\Z)^r$ to $\hat H$ has cardinality $O_\eps(1)$.  By the preceding discussion, we also see that every element of $\pi(\Sigma)$ has order $O_\eps(1)$.  We conclude that the group $\langle \pi(\Sigma) \rangle$ generated by $\pi(\Sigma)$ is a subgroup of $\hat H$ of cardinality $O_\eps(1)$.

Let $H'$ be the orthogonal complement of $\langle \pi(\Sigma) \rangle$, then $|H'| \gg_\eps |H|$.  Now let $l$ be a multiple of $M$ depending on $\eps$ to be chosen later, let $\eps' > 0$ be a small quantity depending on $l,\eps$ to be chosen later, and $m$ to be a large integer depending on $\eps', l, \eps$ to be chosen later.  We let $(h,v)$ be an element of $H'+P_l$.  The number of representations of $(h,v)$ of the form $a_1+\ldots+a_m-b_1-\ldots-b_m$ with $a_1,\ldots,b_m \in A$ can be expressed via Fourier analysis as
$$ |H|^{2m} \sum_{\xi \in \hat H} \int_{(\R/\Z)^r} |\hat 1_A(\xi,\eta)|^{2m} e( \xi \cdot h ) e( \eta \cdot v )\ d\eta.$$
As in Lemma \ref{fstp} (decomposing into the regions $(\xi,\eta) \in \{0\} \times R_{\eps'}$, $(\xi,\eta) \in \Sigma \backslash (\{0\} \times R_{\eps'})$, and $(\xi,\eta) \not \in \Sigma$) one can show that this expression is $\gg_{\eps,\eps'} |A|^{2m}$ if $m$ is large enough, and the claim follows.
\end{proof}

\section{A Balog-Szemer\'edi type lemma}

Given a bipartite graph $G = (V,W,E)$ between two vertex sets $V, W$ (thus $E$ is a subset of $V \times W$), we define the \emph{edge density} $\delta$ of the graph as $\delta := |E|/|V| |W|$, and say that the graph is \emph{$\eps$-regular} for some $\eps > 0$ if we have
$$ \left| |E \cap (V' \times W')| - \delta |V'| |W'| \right| \leq \eps |V| |W|$$
for all $V' \subset V$ and $W' \subset W$, or equivalently if
\begin{equation}\label{vwd}
 |\E_{v \in V, w \in W} (1_E(v,w) - \delta) f(v) g(w)| \leq \eps
\end{equation}
for all $f, g$ bounded between $0$ and $1$. 

We recall a $k$-partite formulation of the famous Szemer\'edi regularity lemma:

\begin{lemma}[Szemer\'edi regularity lemma]\label{reg-lem}  Let $V_1,\ldots,V_k$ be disjoint finite sets, and for each $1 \leq i < j \leq k$ let $E_{ij} \subset V_i \times V_j$ be a bipartite graph connecting $V_i$ and $V_j$.  Let $\eps > 0$.  Then for each $1 \leq i \leq k$ we can partition $V_i = V_{i,1} \cup \ldots V_{i,M_i}$ with $M_i = O_{k,\eps}(1)$, where
\begin{itemize}
\item For all $1 \leq \alpha,\beta \leq M_i$, we have $|V_{i,\alpha}| \sim |V_{i,\beta}|$ (and in particular, $|V_{i,a}| \sim |V_i|/M_i \sim_{k,\eps} |V_i|$);
\item For all $1 \leq i < j \leq k$, and for $1-O(\eps)$ of the pairs of integers $1 \leq \alpha \leq M_i$ and $1 \leq \beta \leq M_j$, the restriction of $E_{ij}$ to $V_{i,\alpha}$, $V_{j,\beta}$ is $\eps$-regular.
\end{itemize}
\end{lemma}

We now conclude a Balog-Szemer\'edi type result in noncommutative groups.

\begin{proposition}[Balog-Szemer\'edi type lemma]\label{bz-prop}  Let $A$ be a finite non-empty subset of a multiplicative group $G = (G,\cdot)$ such that $|A \cdot A^{-1}| \leq K |A|$ for some $K > 0$, let $k_0 \geq 1$ be an integer, and let $\eps > 0$.  Then there exists a subset $A'$ of $A \cdot A^{-1}$ with $|A'| \gg_{K,k_0,\eps} |A|$ such that for every $1 \leq k \leq k_0$, there are at least $(1-\eps) |A'|^{2k}$ tuples $(a_1,\ldots,a_{2k}) \in (A')^{\otimes 2k}$ such that
$$ a_1 a_2^{-1} a_3 a_4^{-1} \ldots a_{2k-1} a_{2k}^{-1} \in A \cdot A^{-1}.$$
\end{proposition}

\begin{proof}  We fix $k_0, K$ and allow all implied constants to depend on these quantities.

From the Cauchy-Schwarz inequality we have
$$ |\{ (a_1,a_2,a_3,a_4) \in A^{\otimes 4}: a_1 a_2^{-1} = a_3 a_4^{-1} \}| \geq \frac{|A|^2}{|A \cdot A^{-1}|} \gg |A|^3.$$
We thus conclude the existence of a set $D \subset A \cdot A^{-1}$ of ``popular quotients'' with the property that $|D| \gg |A|$, and such that every $d \in D$ has at least $\gg |A|$ representations of the form $d = a_1 a_2^{-1}$ with $a_1, a_2 \in A$.

Fix $D$.  We set $V_0 := D$, $V_1 := A$, $V_2 := A$, and define the bipartite graphs $E_{01} \subset V_0 \times V_1$, $E_{02} \subset V_0 \times V_2$ by declaring $(d,a_1) \in V_0 \times V_1$ and $(d,a_2) \in V_0 \times V_2$ whenever $d \in D, a_1, a_2 \in A$ are such that $d = a_1 a_2^{-1}$.  From the preceding discussion we see that $E_{01}, E_{02}$ have edge density $\gg 1$.  In fact, every $d \in D$ is connected via $E_{01}$ to $\gg |A|$ vertices in $V_1$, and connected via $E_{02}$ to $\gg |A|$ vertices in $V_2$.

Let $\eps_1 > 0$ be a small quantity (depending on $k_0,K$) to be chosen later, and let $\eps_2 > 0$ be an even smaller quantity (depending on $k_0,K,\eps_1$) to be chosen later, and so forth up to $\eps_{2k_0+1}$, which is a very small quantity depending on all previous quantities.  Applying Lemma \ref{reg-lem}, we may find partitions $V_i = V_{i,1} \cup \ldots \cup V_{i,M_i}$ for $i=0,1,2$ with $M_i = O_{\eps_2}(1)$ such that $V_{i,\alpha} \sim |A| / M_i$ for all $1 \leq \alpha \leq M_i$, and such that for each $i=1,2$, that the restriction of $E_{0i}$ to $V_{0,\alpha} \times V_{i,\beta_i}$ is $\eps_{2k_0+1}$-regular for $1-O(\eps_2)$ of all $1 \leq \alpha \leq M_0$, $1 \leq \beta_i \leq M_i$.

Applying Markov's inequality, we may find $1 \leq \alpha_0 \leq M_0$ such that for each $i=1,2$, the restriction of $E_{0i}$ to $V_{0,\alpha} \times V_{i,\beta_i}$ is $\eps_{2k_0+1}$-regular for $1-O(\eps_{2k_0+1})$ of all $1 \leq \beta_i \leq M_i$.  The set $V_{0,\alpha}$ will ultimately be our choice for the set $A'$.

Fix $\alpha_0$ as above.  If $i=1,2$, $1 \leq j \leq 2k_0$, and $1 \leq \beta_i \leq M_i$, let us call $\beta_i$ \emph{$(i,j)$-good} if the restriction of $E_{0i}$ to $V_{0,\alpha} \times V_{i,\beta_i}$ is $\eps_{2k_0+1}$-regular and has edge density at least $\eps_j$, and \emph{$(i,j)$-bad} otherwise.  Observe that the total number of edges in $E_{0i}$ between $V_{0,\alpha_0}$ and $(i,j)$-bad cells $V_{i,\beta_i}$ is $\ll \eps_j |V_{0,\alpha_0}| |A|$.  Thus (by Markov's inequality), if we pick $d \in V_{0,\alpha_0}$ uniformly at random, then with probability $1 - O(\eps_j^{1/2})$, there are at most $O(\eps_j^{1/2} |A|)$ edges between $d$ and $(i,j)$-bad cells. 

On the other hand, the total number of edges in $E_{0i}$ between $V_{0,\alpha_0}$ and $V_i$ is $\gg |V_{0,\alpha_0}| |A|$.  Thus we see that there are $\gg M_i$ cells that are $(i,j)$-good for each $1 \leq j \leq 2k_0$.

Let us now fix $1 \leq k \leq k_0$, and pick $d_1,\ldots,d_{2k} \in V_{0,\alpha_0}$ uniformly and independently at random.  By the above discussion and the union bound, we see that with probability $1 - O(\eps_1^{1/2})$, we have that there are at most $O(\eps_j^{1/2} |A|)$ number of edges between $d_l$ and $(i,j)$-bad cells for any $1 \leq j,l \leq 2k$ and $i=1,2$.  Let us now condition on this event, which we will call $E$.

For each $a_1 \in A$, and define $a_2,\ldots,a_{2k+1} \in G$ recursively by solving the equations
\begin{align*}
d_1 &= a_1 a_2^{-1} \\
d_2 &= a_3 a_2^{-1} \\
d_3 &= a_3 a_4^{-1} \\
d_4 &= a_5 a_4^{-1} \\
&\vdots\\
d_{2k} &= a_{2k+1} a_{2k}^{-1}.
\end{align*}
Observe that for fixed $d_1,\ldots,d_{2k}$, the $a_2,\ldots,a_{2k+1}$ are determined uniquely by $a_1$, and the maps $a_1 \to a_j$ are injective for $j=2,\ldots,2k+1$.  Furthermore, each $a_j$ depends only on $a_1$ and $d_1,\ldots,d_{j-1}$.

For each $1 \leq j \leq 2k+1$, we let $f_j(d_1,\ldots,d_{j-1})$ denote the number of $a_1$ such that $a_1,\ldots,a_j$ all lie in $A$.  We now claim that for each $1 \leq j \leq 2k+1$, that we have
\begin{equation}\label{fidge}
 f_j(d_1,\ldots,d_{j-1}) \gg_{\eps_1,\ldots,\eps_{j-1}} |A|
 \end{equation}
with probability $1 - O(\eps_1^{1/2})$.

We prove this by induction on $j$.  When $j=1$ we have $f_1() = |A|$ and the claim is clear, so suppose now that $2 \leq j \leq 2k+1$ and that the claim has already been proven for $j-1$.  To fix the notation we shall assume that $j$ is even; the arguments for odd $j$ are virtually identical and are left to the reader.

By induction hypothesis, we may already condition $d_1,\ldots,d_{j-2}$ so that
$$ f_{j-1}(d_1,\ldots,d_{j-2}) \gg_{\eps_1,\ldots,\eps_{j-2}} |A|.$$
Thus there exist $\gg_{\eps_1,\ldots,\eps_{j-2}} |A|$ choices of $a_1$ (and hence $a_2,\ldots,a_{j-1}$) such that $a_1,\ldots,a_{j-1}$ all lie in $A$.

On the other hand, we know (because of our conditioning to $E$) that $d_{j-2}$ is connected to at most $O(\eps_{j-1}^{1/2} |A|)$ vertices in $(2,j-1)$-bad cells.  Thus, by deleting those vertices from consideration, we obtain $\gg_{\eps_1,\ldots,\eps_{j-2}} |A|$ choices for $a_1,\ldots,a_{j-1} \in A$ such that $a_{j-1}$ lies in a $(2,j-1)$-good cell.

On the other hand (again by our conditioning to $E$), the random variable $d_{j-1}$ is uniformly distributed on a subset of $|V_{0,\alpha_0}|$ of density $\gg 1$, even after conditioning on $d_1,\ldots,d_{j-2}$.  Using the regularity and density of the $(2,j-1)$-cells, we conclude that with probability $1-O(\eps_1^{1/2})$, that $\gg \eps_{j-1}$ of the $a_{j-1}$ listed above are connected via $E_{02}$ to $d_{j-1}$.  By definition of $E_{02}$ and $a_j$ (and the hypothesis that $j$ is even), this implies that $a_j \in A$.  Thus we have \eqref{fidge} with probability $1-O(\eps_1^{1/2})$, closing the induction.

Applying \eqref{fidge} with $j=2k+1$, we conclude in particular that for $1-O(\eps_1^{1/2})$ of all tuples $(d_1,\ldots,d_{2k}) \in V_{0,\alpha_0}^{\otimes 2k}$, that there exists at least one choice of $a_1$ (and hence $a_2,\ldots,a_{2k+1}$) such that $a_1,\ldots,a_{2k+1} \in A$.  Applying the telescoping identity
$$ d_1 d_2^{-1} d_3 d_4^{-1} \ldots d_{2k}^{-1} = a_1 a_{2k+1}^{-1}$$
we conclude that
$$ d_1 d_2^{-1} d_3 d_4^{-1} \ldots d_{2k}^{-1} \in A \cdot A^{-1}.$$
Setting $A' := V_{0,\alpha_0}$ (and choosing $\eps_1$ sufficiently small depending on $\eps$), we obtain the claim.
\end{proof}

We can strengthen the above result slightly by ensuring $A'$ is centred, at the (necessary) cost of now placing $A'$ in $A \cdot A^{-1}$ rather than $A$:

\begin{proposition}[Balog-Szemer\'edi type lemma, again]\label{bz-prop2}  Let $A$ be a finite non-empty subset of a multiplicative group $G = (G,\cdot)$ such that $|A \cdot A^{-1}| \leq K |A|$ for some $K > 0$, let $k_0 \geq 1$ be an integer, and let $\eps > 0$.  Then there exists a centred set $D \subset A \cdot A^{-1}$ with $|D| \gg_{K,k_0,\eps} |A|$ such that for every $1 \leq k \leq k_0$, there are at least $(1-\eps) |D|^{k}$ tuples $(d_1,\ldots,d_k) \in D^{\otimes k}$ such that
$$ d_1 \ldots d_k \in A \cdot A^{-1}.$$
\end{proposition}

\begin{proof}  We may assume that $|A|$ is large depending on $K, k_0, \eps$, since otherwise we may just take $D = \{1\}$.
Observe that it will suffice to construct a symmetric set $D$ rather than a centred set $D$ with the desired properties, since we can convert a symmetric set into a centred one simply by adding $\{1\}$, and the properties of $D$ do not change significantly (adjusting $\eps$ if necessary).

Let $\eps' > 0$ be a small number depending on $k_0, \eps$ to be chosen later.  Again, we may assume $|A|$ large depending on $K,k_0,\eps,\eps'$.  Applying Proposition \ref{bz-prop}, we may find a subset
$A'$ of $A \cdot A^{-1}$ with $|A'| \gg_{K,k_0,\eps'} |A|$ such that for every $1 \leq k \leq k_0$, there are at least $(1-\eps') |A'|^{2k}$ tuples 
$(a_1,\ldots,a_{2k}) \in (A')^{\otimes 2k}$ such that
\begin{equation}\label{asot}
 a_1 a_2^{-1} a_3 a_4^{-1} \ldots a_{2k-1} a_{2k}^{-1} \in A \cdot A^{-1}.
 \end{equation}

We now use the probabilistic method.  Let $a_1,\ldots,a_{2|A'|}$ be elements of $A'$ drawn uniformly and independently at random, and let $d_1,\ldots,d_{2|A'|} \in A \cdot A^{-1}$ be the quantities
$$ d_j := a_{2j} a_{2j+1}^{-1}; \quad d_{j+|A'|} := a_{2j+1} a_{2j}^{-1}$$
and let $D := \{d_1,\ldots,d_{2|A'|}\}$.  Then $D$ is clearly a symmetric subset of $A \cdot A^{-1}$ with $|D| \leq 2|A'|$.

For each $d \in D$, let $\mu(d)$ be the number of representations $d = d_j$ of $d$, where $1 \leq j \leq 2|A'|$, thus
\begin{equation}\label{doo}
\sum_{d \in D} \mu(d) = 2|A'|.
\end{equation}
Observe that if $1 \leq j, j' \leq 2|A'|$ are such that $|j-j'| \neq 0, |A'|$, then $d_j, d_{j'}$ are independent, and the probability that $d_j=d_{j'}$ is at most $1/|A'|$.  Summing over all $j,j'$ (treating the exceptional cases $|j-j'| = 0,|A'|$ separately) and rearranging, one obtains that
$$ \E \sum_{d \in D} \mu(d)^2 \ll |A'|.$$
Thus by Markov's inequality, with probability at least $0.9$, we have
\begin{equation}\label{mark-1}
\sum_{d \in D} \mu(d)^2 \ll |A'|.
\end{equation}
which by Cauchy-Schwarz and \eqref{doo} implies that $|D| \sim |A'|$.

Next, observe that if $1 \leq k \leq k_0$ and $1 \leq j_1,\ldots,j_k \leq 2|A'|$ are such that no two of the $j_i, j_{i'}$ are equal or differ by $|A'|$, then $d_{j_1},\ldots,d_{j_k}$ are independent, and by \eqref{asot} we see that
$$ d_{j_1} \ldots d_{j_k} \in A \cdot A^{-1}$$
with probability $1-O(\eps')$.  Summing over all choices of $k$ and $j_1,\ldots,j_k$ (again treating the exceptional cases separately), we see that
$$ \E \sum_{k=1}^{k_0} \frac{1}{|A'|^k} |\{ 1 \leq j_1,\ldots,j_k \leq 2|A'|: d_{j_1} \ldots d_{j_k} \not \in A \cdot A^{-1} \} \ll_{k_0} \eps'.$$
Thus by Markov's inequality, with probability at least $0.9$, we have
$$ \sum_{k=1}^{k_0} \frac{1}{|A'|^k} |\{ 1 \leq j_1,\ldots,j_k \leq 2|A'|: d_{j_1} \ldots d_{j_k} \not \in A \cdot A^{-1} \} \ll_{k_0} \eps'$$
and thus
\begin{equation}\label{mark-2}
 |\{ 1 \leq j_1,\ldots,j_k \leq 2|A'|: d_{j_1} \ldots d_{j_k} \not \in A \cdot A^{-1} \}| \ll_{k_0} \eps' |A'|^k
 \end{equation}
for all $1 \leq k \leq k_0$.  Thus with probability at least $0.8$, \eqref{mark-1} and \eqref{mark-2} both hold.  We now select $a_1,\ldots,a_{2|A'|}$ so that this is the case.

We rearrange \eqref{mark-2} as
$$ \sum_{d_1,\ldots,d_k \in D: d_1 \ldots d_k \not \in A \cdot A^{-1}} \mu(d_1) \ldots \mu(d_k) \ll_{k_0} \eps' |A'|^k \ll_{k_0} \eps' |D|^k.$$
In particular this implies that
$$ \{ (d_1,\ldots,d_k) \in D^{\otimes k}: d_1 \ldots d_k \not \in A \cdot A^{-1} \} \ll_{k_0} \eps' |D|^k$$
and the claim follows by choosing $\eps'$ sufficiently small depending on $k_0, \eps$.
\end{proof}

\begin{remark}  In the case when $G = (G,+)$ is abelian, we may invoke Theorem \ref{gr2} and obtain a stronger result of a similar flavour, namely that $\pm 4A$ contains a coset progression $H+P$ of size comparable to $A$, and thus also contains a set of the form $kA'-kA'$ for some $A'$ of size comparable to $A$, for any fixed $k$.  It is thus reasonable to conjecture an analogous statement in the non-commutative case; for instance, if $A$ is a $K$-approximate group, and $k \geq 1$, one would expect the existence of a set $A'$ of size $\gg_{K,k} |A|$ such that $(A')^k \subset A^{100}$ (say).  Unfortunately our methods do not seem to establish such a result.
\end{remark}

\begin{remark} Because of our reliance on the Szemer\'edi regularity lemma, the implicit bounds in the above results are extremely poor (of tower-exponential type).  In view of the polynomial strengthening of the Balog-Szemer\'edi theorem \cite{balog} due to Gowers\cite{gowers-4}, it is natural to expect these bounds to be improvable here also.  Note however that the best bounds for the Green-Ruzsa results are still exponential in $K$, and so we do not know how to obtain a polynomial bound in the above results even in the abelian case.\footnote{Note added in proof: in \cite{croot}, \cite{sanders-2} a significantly strengthened version of Proposition \ref{bz-prop} was established, in which the $\eps$ error was eliminated, and the tower-exponential bounds were improved to exponential bounds by avoiding an appeal to the regularity lemma.  However, the methods in these papers appear to still only give exponential bounds even after assuming a polynomial version of Freiman's theorem, so further techniques would still be needed to reach the polynomial case.}
\end{remark}

\end{document}